\documentclass[letter]{amsart}

\usepackage{amssymb}
\usepackage{amsfonts}
\usepackage{amsmath}
\usepackage{graphicx}
\usepackage{mathrsfs}
\usepackage{dsfont}
\usepackage{amscd}
\usepackage{multirow}
\usepackage{mathpazo}
\usepackage[all]{xy}
\usepackage[scaled]{helvet} % ss
\usepackage{courier} % ttcompact operators fejer kernel finite dimensional approximation harmonic basis
\normalfont
\usepackage{calligra}
\usepackage[T1]{fontenc}
\usepackage{verbatim}
\usepackage{hyperref}

\newcommand{\N}{{\mathds{N}}}
\newcommand{\Z}{{\mathds{Z}}}

\newcommand{\R}{{\mathds{R}}}
\newcommand{\C}{{\mathds{C}}}
\newcommand{\T}{{\mathds{T}}}
\newcommand{\U}{{\mathds{U}}}
\newcommand{\D}{{\mathfrak{D}}}
\newcommand{\A}{{\mathfrak{A}}}
\newcommand{\B}{{\mathfrak{B}}}

\newcommand{\Nbar}{\overline{\N}_\ast}
\newcommand{\Lip}{{\mathsf{L}}}
\newcommand{\Hilbert}{{\mathscr{H}}}
\newcommand{\Diag}[2]{{\operatorname*{Diag}\left[{#1}\middle| {#2} \right]}}

\newcommand{\propinquity}{{\mathsf{\Lambda}}}
\newcommand{\Kantorovich}[1]{{\mathsf{mk}_{#1}}}

\newcommand{\Haus}[1]{{\mathsf{Haus}_{#1}}}

\newcommand{\StateSpace}{{\mathscr{S}}}

\newcommand{\mongekant}{{Mon\-ge-Kan\-to\-ro\-vich metric}}

\newcommand{\Lqcms}{{\JLL} quantum compact metric space}

\newcommand{\unit}{1}

\newcommand{\sa}[1]{{\mathfrak{sa}({#1})}}

\newcommand{\Adm}{{\mathrm{Adm}}}

\newcommand{\inner}[2]{{\left<{#1},{#2}\right>}}
\newcommand{\JLL}{Lei\-bniz}
\newcommand{\tr}{{\operatorname*{tr}}}
\newcommand{\Mod}[2]{{{\raisebox{.2em}{$#1$}\left/\raisebox{-.2em}{$#2$}\right.}}}

\newcommand{\cyclic}[2]{{ \left[{#1} \mod I_{#2} \right]  }}
\newcommand{\cycle}[2]{{ \left[{#1} \middle| I_{#2} \right]  }}

\newcommand{\dom}[1]{{\operatorname*{dom}({#1})}}
\newcommand{\diam}[2]{{\mathrm{diam}\left({#1},{#2}\right)}}

\newcommand{\bridgeset}[2]{{\text{\calligra Bridges}\left({#1}\rightarrow{#2}\right)}}

\newcommand{\bridgereach}[2]{{\varrho\left({#1}\middle|{#2}\right)}}
\newcommand{\bridgeheight}[2]{{\varsigma\left({#1}\middle|{#2}\right)}}
\newcommand{\bridgelength}[2]{{\lambda\left({#1}\middle|{#2}\right)}}
\newcommand{\treklength}[1]{{\lambda\left({#1}\right)}}

\newcommand{\bridgenorm}[2]{{\mathsf{bn}_{ {#1}  }\left({#2}\right)}}

\newcommand{\Jordan}[2]{{{#1}\circ{#2}}} %{{{#1}\circledcirc{#2}}}
\newcommand{\Lie}[2]{{\left\{{#1},{#2}\right\}}} %{{\left[{#1},{#2}\right]_\ast}}

\newcommand{\LCQMS}{{\mathcal{L}^\ast}}

\theoremstyle{plain}
\newtheorem{theorem}{Theorem}[subsection]

\newtheorem{claim}[theorem]{Claim}

\newtheorem{corollary}[theorem]{Corollary}

\newtheorem{lemma}[theorem]{Lemma}

\newtheorem{theorem-definition}[theorem]{Theorem-Definition}

\theoremstyle{definition}
\newtheorem{definition}[theorem]{Definition}

\newtheorem{notation}[theorem]{Notation}

\theoremstyle{remark}

\newtheorem{remark}[theorem]{Remark}

\renewcommand{\geq}{\geqslant}
\renewcommand{\leq}{\leqslant}

\newcommand{\conv}[1]{{*_{#1}}} %{{\widehat{\star_{#1}}}}

\numberwithin{equation}{subsection}

\begin{document}

%\frontmatter

\title[Quantum Tori and the Quantum Gromov-Haus\-dorff Propinquity]{Convergence of Fuzzy Tori and Quantum Tori for the quantum Gromov-Hausdorff Propinquity: an explicit approach}
\author{Fr\'{e}d\'{e}ric Latr\'{e}moli\`{e}re}
\email{frederic@math.du.edu}
\urladdr{http://www.math.du.edu/\symbol{126}frederic}
\address{Department of Mathematics \\ University of Denver \\ Denver CO 80208}

\date{\today}
\subjclass[2000]{Primary:  46L89, 46L30, 58B34.}
\keywords{Noncommutative metric geometry, quantum Gromov-Haus\-dorff distance, Monge-Kantorovich distance, Quantum Metric Spaces, Lip-norms, compact C*-metric spaces, Leibniz seminorms, Quantum Tori, Finite dimensional approximations.}

\begin{abstract}
Quantum tori are limits of finite dimensional C*-algebras for the quantum Gromov-Hausdorff propinquity, a metric defined by the author as a strengthening of Rieffel's quantum Gromov-Hausdorff designed to retain the C*-algebraic structure. In this paper, we propose a proof of the continuity of the family of quantum and fuzzy tori which relies on explicit representations of the C*-algebras rather than on more abstract arguments, in a manner which takes full advantage of the notion of bridge defining the quantum propinquity.
\end{abstract}

\maketitle

%\tableofcontents

% \let\oldtocsection=\tocsection
% \let\oldtocsubsection=\tocsubsection
% \let\oldtocsubsubsection=\tocsubsubsection
% \renewcommand{\tocsection}[2]{\hspace{0em}\oldtocsection{#1}{#2}}
% \renewcommand{\tocsubsection}[2]{\hspace{1em}\oldtocsubsection{#1}{#2}}
% \renewcommand{\tocsubsubsection}[2]{\hspace{2em}\oldtocsubsubsection{#1}{#2}}
% \setcounter{tocdepth}{4}
% \tableofcontents

\newcommand{\physics}[1]{#1}

%%%%%%%%%%%%%%%%%%%%%%%%%%%%%%%%%%%%%%%%%%%%%%%%%%%%%%%%%%%%%%%%%%%%%%%%%%%%%%%%%%%%%%%%%%%%%%%%%%%%%%%%%

%\mainmatter

\section{Introduction}

The quantum Gromov-Hausdorff propinquity and the dual Gromov-Hausdorff propinquity, introduced by the author in \cite{Latremoliere13} and \cite{Latremoliere13b}, are metrics on the class of {\Lqcms s} up to isometric isomorphism, designed with the goal of providing a tool to work within the C*-algebraic framework in noncommutative metric geometry \cite{Connes89,Connes,Rieffel98a,Rieffel99,Rieffel00,Rieffel05}. Our metrics address important issues raised by recent work in noncommutative metric geometry \cite{Rieffel09,Rieffel10,Rieffel10c,Rieffel11,Rieffel12}, where the C*-algebraic structure and the Leibniz property occupy an ever larger role, yet did not fit nicely within the framework of the quantum Gromov-Hausdorff distance \cite{Rieffel00}. As our new metrics are stronger than the quantum Gromov-Haus\-dorff distance, a question of central interest is to prove that quantum tori are still limits of fuzzy tori in the spirit of \cite{Latremoliere05}, and still form a continuous family for our new metrics.

A positive answer to these questions can be gleaned from the work of Kerr and Li in \cite{kerr09}, where the unital nuclear distance, which dominates our propinquity metrics, can be used together with our earlier work \cite{Latremoliere05} on finite dimensional approximations of quantum tori, to obtain that quantum and fuzzy tori form a continuous family for the quantum and the dual Gromov-Hausdorff propinquity \cite[Theorem 6.8]{Latremoliere13}. However, the very powerful tool of subtrivialization of fields of C*-algebras, developed by Haagerup and R{\o}rdam \cite{Rordam95} and Blanchard \cite{Blanchard97}, is employed in \cite{li03,kerr09}, which relies on somewhat abstract constructions. Thus, it is difficult to provide an explicit admissible Leibniz Lip-norm, in the sense of \cite{Rieffel00}, from such a construction.

It would be very interesting to obtain explicit Leibniz Lip-norms to derive the continuity of the quantum tori for our quantum Gromov-Hausdorff propinquity \cite{Latremoliere13}, for a variety of reasons. First, the quantum tori and their finite dimensional counter parts are well-studied and understood, and one would like to provide a more concrete construction of convergence for such a central and accessible example. Second, the original physical motivation behind the search for finite dimensional approximations of quantum tori \cite{Connes97,Seiberg99,tHooft02,Zumino98} would benefit from more explicit constructions which may provide useful quantitative, rather than qualitative methods for convergence. Third, future explorations of the continuity of various C*-algebraic structures may well be easier if one can work with such natural objects as the left regular representations, trace class operators and explicit Lip-norms rather than more implicit constructions. Such explorations are in large part behind the reason for introducing the quantum and dual Gromov-Hausdorff propinquity as well-behaved substitutes for the quantum proximity of Rieffel \cite{Rieffel10c}, which was in turn motivated by the study convergence of projective modules associated to convergent sequences of {\Lqcms s} for the quantum Gromov-Hausdorff distance \cite{Rieffel09}. Last, our method may provide a path to obtain quantitative results about Gromov-Hausdorff convergence for other continuous field of C*-algebras, even when a more generic argument may provide an abstract convergence result. We were thus curious as to whether one can explicit natural Leibniz Lip-norms in the convergence of the quantum tori in the spirit of \cite{Rieffel09,Rieffel10c,Rieffel11}, and this paper provides the result of our investigation.

We propose a proof of the convergence of fuzzy tori to quantum tori for the quantum Gromov-Hausdorff propinquity where the bridges are given explicitly, in terms of the left-regular representations of the various twisted C*-algebras involved and an explicit choice of a pivot. Our proof thus does not rely on subtrivialization arguments, and may thus be more amenable to quantitative considerations. Our purpose, in particular, is the construction of Leibniz Lip-norms to replace the non-Leibniz Lip-norms of \cite{Rieffel00, Latremoliere05}, in a descriptive manner, rather than just proving the existence of such Lip-norms without a grasp of their actual form. Thus, our proof provides quantitative, rather than qualitative, tools to work with convergence of quantum tori and fuzzy tori.

Our paper begins with a brief survey of the quantum Gromov-Hausdorff propinquity, which is the metric we will use in this paper. Since this metric dominates our dual Gromov-Hausdorff propinquity, our result in this paper also applies to the later. We refer to \cite{Latremoliere13,Latremoliere13b} for a more comprehensive exposition of the propinquity metrics, including motivations are relations to other metric defined on classes of quantum compact metric spaces. We then survey results which we shall need regarding the quantum tori, with a focus on the noncommutative metric aspects. We further perform various computations using the left regular representations of the quantum and fuzzy tori, as a preliminary step for the last section, where we construct our explicit bridges for the quantum propinquity, including explicit Leibniz Lip-norms, and use them to prove the convergence of fuzzy tori to quantum tori.

\section{Quantum Gromov-Hausdorff Propinquity}

As an informal motivation for our work and the introduction of the quantum Gromov-Hausdorff propinquity, we begin this section with the problem which this paper solves. For all $n\in\N$, let us be given a complex numbers $\rho_n$ such that $\rho_n^n = 1$, and let us define the two $n\times n$ unitary matrices:
\begin{equation*}
U_n = \begin{pmatrix}
0 & \hdots &        & & &1 \\
1 & 0      & \hdots & & &  \\
\ddots & \ddots &   & & &  \\
  &        &        & &1& 0\\ 
\end{pmatrix}
\text{ and }
V_n = \begin{pmatrix}
1 &      & & & &\\
  & \rho_n & & & &\\
  & & \rho_n^2 & & &\\
  & & & \ddots & & \\
  & & & & & \rho_n^{n-1}
\end{pmatrix}\text{.}
\end{equation*}
By construction, $U_nV_n = \rho_n V_n U_n$. Such pair of matrices appear in the literature in mathematical physics as well as quantum information theory, among others. The C*-algebras $C^\ast(U_n,V_n)$ are sometimes called fuzzy tori. Often, a desirable outcome of some computations carried over fuzzy tori is that one can obtain interesting results when $n$ goes to infinity under the condition that the sequence $(\rho_n)_{n\in\N}$ converges --- examples of such situations are found in the mathematical physics literature, for instance \cite{Connes97,Seiberg99,tHooft02,Zumino98}, to cite but a few. Informally, one would expect that the limit of the fuzzy tori $C^\ast(U_n,V_n)$ would be the universal C*-algebra  $C^\ast(U,V)$ generated by two unitaries $U$ and $V$ subject to the relation $UV = \rho VU$ where $\rho = \lim_{n\rightarrow\infty} \rho_n$, i.e. a quantum torus. Yet, as quantum tori are not AF --- for instance, their $K_1$ groups are nontrivial --- making sense of such a limiting process is challenging. Rieffel's quantum Gromov-Hausdorff distance \cite{Rieffel00} provides a first framework in which such a limit can be justified \cite{Latremoliere05}. However, this distance may be null between *-isomorphic C*-algebras: in other words, it does not capture the C*-algebraic structure fully. We introduced in \cite{Latremoliere13} a metric which does capture the C*-algebraic structure, called the quantum propinquity, for which we propose to study the convergence of fuzzy tori to quantum tori by providing an explicit construction, more amenable to computational methods, rather than relying on a more abstract result \cite{kerr09}. The motivation to seek an explicit proof of convergence for such a metric is that, as the field of noncommutative metric geometry evolves, one is interested in studying consequences of convergence for various C*-algebra related structures, such as modules \cite{Rieffel10}, as is done in the physics literature outside of any formal construction. For such computations, it becomes desirable to work within the class of C*-algebras (rather than order unit spaces) and use well-understood constructs, such as the left regular representations, rather than more complicated abstractions. We propose in this section to summarize the construction of the quantum propinquity, and then provide such an explicit proof of convergence in the remainder of the paper.

\subsection{Quantum Gromov-Hausdorff Distance}

Noncommutative metric geometry \cite{Connes89,Connes,Rieffel98a,Rieffel99,Rieffel00} proposes to study noncommutative generalizations of Lipschitz algebras \cite{Weaver99}, defined as follows.

\begin{notation}
Let $\A$ be a unital C*-algebra. The norm of $\A$ is denoted by $\|\cdot\|_\A$ and the unit of $\A$ is denoted by $\unit_\A$, while the state space of $\A$ is denoted by $\StateSpace(\A)$ and the space of self-adjoint elements in $\A$ is denoted by $\sa{\A}$.
\end{notation}

\begin{definition}\label{quantum-compact-metric-space-def}
A \emph{quantum compact metric space} $(\A,\Lip)$ is a pair of a unital C*-algebra $\A$, with unit $\unit_\A$, and a seminorm $\Lip$ defined on a dense subset $\dom{\Lip}$ of the self-adjoint part $\sa{\A}$ of $\A$, such that:
\begin{enumerate}
\item $\{a\in\dom{\Lip} : \Lip(a) = 0 \} = \R \unit_\A$,
\item the metric defined on the state space $\StateSpace(\A)$ of $\A$ by setting for all $\varphi,\psi\in\StateSpace(\A)$:
\begin{equation*}
\Kantorovich{\Lip}(\varphi,\psi) = \sup \left\{ |\varphi(a)-\psi(a)| : a\in\dom{\Lip}\text{ and }\Lip(a)\leq 1 \right\}
\end{equation*}
metrizes the weak* topology on $\StateSpace(\A)$.
\end{enumerate}
The seminorm $\Lip$ is then called a \emph{Lip-norm}, while $\Kantorovich{\Lip}$ is the \emph{\mongekant} associated with $\Lip$.
\end{definition}

As a matter of convention, we extend seminorms on dense subsets of a vector space by setting them to the value $\infty$ outside of their domain. We adopt this practice implicitly in the rest of this paper.

The fundamental examples of quantum compact metric spaces are given by pairs $\left(C(X),\mathsf{Lip}_{\mathsf{m}}\right)$, where $X$ is a compact metric space, $\mathsf{m}$ is a continuous metric on $X$ and, for any $f\in C(X)$, we define:
\begin{equation}\label{Lip-eq}
\mathsf{Lip}_{\mathsf{m}}(f) = \sup\left\{\frac{|f(x)-f(y)|}{\mathsf{m}(x,y)} : x,y\in X,x\not=y \right\}\text{.}
\end{equation}
The restriction of $\mathsf{Lip}_{\mathsf{m}}$ to the space $\sa{C(X)}$ of real-valued continuous functions on $X$ is a Lip-norm. 

Another class of fundamental examples is given by unital C*-algebras on which a compact group, endowed with a continuous length function, acts ergodically \cite{Rieffel98a}; the particular case of quantum and fuzzy tori will be detailed in the next section.

We propose a generalization of the notion of a quantum compact metric space to quantum locally compact metric spaces in \cite{Latremoliere12b}, following on our earlier work in \cite{Latremoliere05b}.

Rieffel introduced in \cite{Rieffel00} a first notion of convergence for quantum compact metric spaces: the quantum Gromov-Hausdorff distance, which is a noncommutative analogue of Gromov's distance between compact metric spaces. A motivation for this metric is to provide a framework for various approximations of C*-algebras found in the mathematical literature (e.g. \cite{Connes97}), as well as a new tool for the study of the metric aspects of noncommutative geometry. 

The construction of the quantum Gro\-mov-Haus\-dorff distance proceeds as follows: given two quantum compact metric spaces $(\A,\Lip_\A)$ and $(\B,\Lip_\B)$, we consider the set $\Adm(\Lip_\A,\Lip_\B)$ of all Lip-norms on $\A\oplus\B$ whose quotients to $\A$ and $\B$ are, respectively, $\Lip_\A$ and $\Lip_\B$. For any such admissible Lip-norm $\Lip$, one may consider the Haus\-dorff distance between $\StateSpace(\A)$ and $\StateSpace(\B)$ identified with their isometric copies in $\StateSpace(\A\oplus\B)$, where the state spaces $\StateSpace(\A)$, $\StateSpace(\B)$ and $\StateSpace(\A\oplus\B)$ are equipped with the {\mongekant} associated to, respectively, $\Lip_\A$, $\Lip_\B$ and $\Lip$. The infimum of these Haus\-dorff distances over all possible choices of $\Lip \in \Adm(\Lip_\A,\Lip_\B)$ is the quantum Gro\-mov-Haus\-dorff distance between $(\A,\Lip_\A)$ and $(\B,\Lip_\B)$.

The notion of admissibility can be understood, in the classical picture, as the statement that a real-valued Lipschitz function $f$ on a compact metric space $X$ can be extended to a real-valued Lipschitz function $\tilde{f}$ with the same Lipschitz seminorm on the disjoint union $X\coprod Y$, where $Y$ is some other compact space and $X\coprod Y$ is metrized by an admissible metric, i.e. a metric such that the canonical embeddings of $X$ and $Y$ in $X\coprod Y$ are isometries. This illustrates why working with self-adjoint elements is desirable, as this extension property fails for complex valued Lipschitz functions \cite{Weaver99}. Now, the Gromov-Hausdorff distance may be defined as the infimum of the Hausdorff distance between $X$ and $Y$ in $X\coprod Y$ taken over all possible admissible metrics on $X\coprod Y$ \cite{Gromov} (there are alternative definitions of the Gromov-Hausdorff distance which lead to equivalent metrics \cite{Gromov}). It is thus clear to see the analogy between Rieffel's quantum Gromov-Hausdorff metric and Gromov's original metric \cite{Gromov81,Gromov}.

In fact, Rieffel's work in \cite{Rieffel99,Rieffel00} takes place within the category of order-unit spaces: all the above definitions are given within this category, and our current description is a special case of Rieffel's original framework where we focus on C*-algebras. The advantage of the generality of Rieffel's approach is that it allows to prove that two quantum compact metric spaces are close using the triangle inequality and various ``pivot'' or intermediate compact quantum metric spaces which are not based upon C*-algebras. To begin with, the self-adjoint part of C*-algebras is not a *-subalgebra, though it is a Jordan-Lei algebra --- a fact we shall take advantage of in our construction of the quantum propinquity. Another prime example, in the case of the continuity of the family of quantum tori, is to truncate the ``Fourier'' series representations of elements of the quantum tori to polynomials in the generators of bounded degrees \cite{Rieffel00}. The resulting order-unit spaces are no longer Jordan-Lie algebras.

This flexibility, however, comes at a price. A first consequence is that two quantum compact metric spaces may be at distance zero for the quantum Gromov-Hausdorff distance without their underlying C*-algebras to be *-isomorphic: distance zero only provides a unital order-isomorphism between their self-adjoint parts, whose dual restricts to an isometry between state spaces endowed with their {\mongekant}. Such a map gives rise to a Jordan isomorphism between the self-adjoint part of the compact quantum metric spaces \cite{Alfsen01}, yet can not distinguish, for instance, between a C*-algebra and its opposite. There are some examples of C*-algebras which are not *-isomorphic to their opposite. 

Second, if one wishes to study convergences of associated C*-algebraic structures such as projective modules, one encounter the difficulty of having to extend these notions to order-unit spaces. Thus, much effort was put to fix one or both of these problems \cite{li03,li05,kerr02,kerr09,Wu06b,Rieffel10c}. The quantum propinquity, which we introduce next, is our own attempt \cite{Latremoliere13} to address these issues. Its particularity is that it focuses on solving, in addition to the two problems cited above, the third complication of working with Leibniz seminorms. This later issue seems very important in recent work \cite{Rieffel10} and our metric is designed to address this matter as well as the coincidence axiom and the desire to work within the category of C*-algebras. We also introduced a dual version, called the dual Gromov-Hausdorff propinquity \cite{Latremoliere13b}, which also addresses the same difficulties, and is complete, though it is more lax in what constitute an admissible Lip-norm. As the dual Gromov-Hausdorff propinquity is dominated by the quantum propinquity, and as our work applies to the stronger metric, we shall focus on the quantum propinquity in this paper.

% The Lipshitz seminorm certainly makes sense for complex valued functions over $X$; however real-valued Lipschitz functions enjoy the property that, given a closed subset $Y$ of $X$, and a real-valued Lipschitz function $f$ on $Y$, there exists an extension of $f$ to $X$ which is Lipchitz with the same Lipschitz seminorm. This property will prove very important in the construction of the quantum Gromov-Hausdorff distance and the quantum Gromov-Hausdorff propinquity. Complex-valued Lipschitz functions do not possess this property --- a factor of $\sqrt{2}$ occurs when extending. This justifies our focus on real-valued Lipschitz functions --- and more generally, on self-adjoint elements. A discussion of the technical difficulties which arise when dealing with complex-valued Lipschitz functions in our context can be found in \cite{Rieffel10c}.

% The Lipschitz seminorm enjoys an important additional property, which connects it to the C*-algebraic structure of $C(X)$: it satisfies the Leibniz identity. 

\subsection{The Quantum Gromov-Hausdorff Propinquity}

This section introduces our metric, the quantum Gromov-Hausdorff propinquity --- which we also call the quantum propinquity for brevity. This section summarizes \cite{Latremoliere13}, to which we refer for discussions, motivations, comparison to other metrics, and most importantly, proofs of the theorems we cite here.

The Lipschitz seminorm $\mathrm{Lip}_{\mathsf{m}}$ associated to a compact metric space $(X,\mathsf{m})$ (see Equation (\ref{Lip-eq}) enjoys a natural property with regard to the multiplication of functions in $C(X)$, called the Leibniz property for seminorms:
\begin{equation}\label{Leibniz-eq}
\forall f,g \in C(X) \quad \mathrm{Lip}_{\mathsf{m}}(fg)\leq\|f\|_{C(X)}\mathrm{Lip}_{\mathsf{m}}(g) + \mathrm{Lip}_{\mathsf{m}}(f)\|g\|_{C(X)}\text{.}
\end{equation}
Moreover, the Lipschitz seminorm is lower-semicontinuous with respect to the C*-norm of $C(X)$, i.e. the uniform convergence norm on $X$. These two additional properties were not assumed in Definition (\ref{quantum-compact-metric-space-def}), yet they are quite natural and prove very useful in developing a C*-algebra based theory of Gromov-Hausdorff convergence, as we shall now see.

While the product of two real-valued functions over $X$ is still real-valued, the situation in the noncommutative setting is more complex. On the other hand, we wish to avoid the difficulties involved in working with the analogues of complex-valued Lipschitz functions (see \cite{Rieffel10c} for some of these difficulties). Thus, we are led to the following generalization of the Leibniz property:
\begin{notation}
For any C*-algebra $\A$ and for any $a,b \in \sa{\A}$, we denote the Jordan product $\frac{1}{2}(ab+ba)$ of $a$ and $b$ by $\Jordan{a}{b}$, while we denote the Lie product $\frac{1}{2i}(ab-ba)$ of $a$ and $b$ by $\Lie{a}{b}$. The triple $(\sa{\A},\Jordan{\cdot}{\cdot},\Lie{\cdot}{\cdot})$ is a Jordan-Lie algebra \cite{Alfsen01,landsman98}.
\end{notation}
\begin{definition}\label{lqcms-def}
A quantum compact metric space $(\A,\Lip)$ is a \emph{\Lqcms} when $\Lip$ is lower semi-continuous for the norm $\|\cdot\|_\A$ of $\A$, and for all $a,b \in \sa{\A}$, we have:
\begin{align*}
\Lip\left(\Jordan{a}{b}\right)&\leq\|a\|_\A\Lip(b) + \Lip(a)\|b\|_\A
\intertext{and}
\Lip\left(\Lie{a}{b}\right)&\leq\|a\|_\A\Lip(b) + \Lip(a)\|b\|_\A\text{.}
\end{align*}
\end{definition}

The purpose of the quantum propinquity is to define a metric on the class $\LCQMS$ of all {\Lqcms s}, up to the following natural notion of isometry:
\begin{definition}\label{isometry-def}
Two {\Lqcms s} $(\A,\Lip_\A)$ and $(\B,\Lip_\B)$ are \emph{isometrically isomorphic} when there exists a *-isomorphism $h$ from $\A$ onto $\B$ such that, for all $a\in\sa{\A}$, we have $\Lip_\B\circ h (a) = \Lip_\A(a)$.
\end{definition}
This notion of isometry is stronger than the one obtained for the quantum Gromov-Hausdorff distance in two ways. First and most notably, it requires that the given isometry is a *-isomorphism. Second, since we assumed that the Lip-norms of {\Lqcms s} are lower-semicontinuous and defined on a Banach space, a unital, positive linear map $h: \sa{\A}\rightarrow\sa{\B}$ has a dual map $\varphi\in\StateSpace(\B)\rightarrow\varphi\circ h\in\StateSpace(\A)$ which is an isometry from $\Kantorovich{\Lip_\B}$ into $\Kantorovich{\Lip_\A}$ if and only if $\Lip_\B\circ h = \Lip_\A$.

A natural first attempt to modify the quantum Gromov-Hausdorff distance to work in $\LCQMS$ is to require that an admissible Lip-norm should be Leibniz, where admissible is used in the sense of the previous section. However, this construction proposed by Rieffel in \cite{Rieffel10c} does not seem to lead to a metric. The difficulty, known for quite some time \cite{Blackadar91}, is that the quotient of a Leibniz seminorm may not be Leibniz; in turn the proof of the triangle inequality for the quantum Gromov-Hausdorff distance fails to specialize to the proximity. Rieffel introduced in \cite{Rieffel10c} the quantum proximity as such a modification of his quantum Gromov-Hausdorff distance, where admissibility is even more constrained (as it requires the strong Leibniz property, which is expressed in terms of the C*-algebra multiplication rather than the Jordan-Lie structure, and adds a condition about invertible elements. The issues are the same and are addressed just as well with our quantum propinquity).

A modification of the proximity led us to define the quantum Gromov-Hausdorff propinquity in \cite{Latremoliere13}, and later the dual Gromov-Hausdorff propinquity in \cite{Latremoliere13b}, which are metrics on $\LCQMS$ with the proper coincidence property --- the latter being also complete. There is an interest in being even more specific in our choice of admissible Lip-norms than require the Leibniz, or even strong Leibniz property. Rieffel proposes to use Leibniz Lip-norms constructed from bimodules, where the bimodules are themselves C*-algebras \cite{Rieffel10c,Rieffel10}. This stronger structural requirement is the basis for the quantum Gromov-Hausdorff propinquity (and is not considered with our dual propinquity).

The construction of the quantum propinquity relies on the notion of a bridge, defined as follows. We refer to \cite{Latremoliere13} for all the following definitions and results.

\begin{definition}\label{bridge-def}
Let $(\A,\Lip_\A)$ and $(\B,\Lip_\B)$ be two {\Lqcms s}. A \emph{bridge} from $(\A,\Lip_\A)$ to $(\B,\Lip_\B)$ is a quadruple $\gamma=(\D,\pi_\A,\pi_\B,\omega)$ where:
\begin{enumerate}
\item $\D$ is a unital C*-algebra,
\item $\pi_\A: \A\rightarrow\D$ and $\pi_\D:\B\rightarrow\D$ are unital *-monomorphisms,
\item $\omega\in\D$ such that the set:
\begin{equation*}
\StateSpace_1(\omega) = \left\{ \varphi \in \StateSpace(\D) : \forall d\in\sa{\D} \quad \varphi(\omega d)=\varphi(d\omega) = \varphi(d) \right\}
\end{equation*}
is not empty.
\end{enumerate}
The C*-algebra $\D$ is the \emph{bridge $C^\ast$-algebra} and $\omega$ is the \emph{pivot element} of the bridge $\gamma$.
\end{definition}

The purpose of bridges is to construct admissible Leibniz Lip-norms. Given a bridge $\gamma = (\D,\pi_\A,\pi_\B,\omega)$ from a {\Lqcms} $(\A,\Lip_\A)$ to a {\Lqcms} $(\B,\Lip_\B)$, we begin by defining the bridge seminorm $\bridgenorm{\gamma}{\cdot,\cdot}$ of $\gamma$:
\begin{equation*}
(a,b)\in\sa{\A\oplus\B} \longmapsto \bridgenorm{\gamma}{a,b} = \left\|\pi_\A(a)\omega - \omega\pi_\B(b)\right\|_\D\text{.}
\end{equation*}
We then wish to find $\lambda > 0$ such that the seminorm:
\begin{equation}\label{admissible-eq}
\Lip_{\gamma,\lambda} : (a,b)\in\sa{\A\oplus\B} \longmapsto \max\left\{ \Lip_\A(a),\Lip_\B(b),\frac{1}{\lambda}\bridgenorm{\gamma}{a,b}\right\}
\end{equation}
is an admissible Lip-norm on $\sa{\A}\oplus\sa{\B}$ for $(\Lip_\A,\Lip_\B)$. By construction, for such a choice of $\lambda>0$, the seminorm $\Lip_{\gamma,\lambda}$ is a Leibniz Lip-norm. We introduce several numerical quantities related to $\gamma$ to determine which $\lambda>0$ could be chosen.

The first number is the \emph{reach} of the bridge, which measures how far two {\Lqcms s} are using, informally, the bridge seminorm:
\begin{notation}
The Hausdorff distance on subsets of a metric space $(X,\mathsf{m})$ is denoted by $\Haus{\mathsf{m}}$. If $(X,\mathsf{m})$ is given as a vector space $X$ and the distance $\mathsf{m}$ induced from a norm $\|\cdot\|_X$ on $X$, the we simply write $\Haus{\|\cdot\|_X}$ for $\Haus{\mathsf{m}}$.
\end{notation}

\begin{definition}\label{reach-def}
The \emph{reach} $\bridgereach{\gamma}{\Lip_\A,\Lip_\B}$ of a bridge $\gamma = (\D,\pi_\A,\pi_\B,\omega)$ from a {\Lqcms} $(\A,\Lip_\A)$ to a {\Lqcms} $(\B,\Lip_\B)$ is:
\begin{equation*}
\Haus{\|\cdot\|_\D}\left(\pi_\A\left(\left\{a\in\sa{\A}:\Lip_\A(a)\leq 1\right\}\right)\omega,\omega\pi_\B\left(\left\{b\in\sa{\B}:\Lip_\B(b)\leq 1\right\}\right)\right)\text{.}
\end{equation*}
\end{definition}
While the reach is defined as the Hausdorff distance between non-compact, not bounded subsets, one can verify that the reach is always finite and reached \cite{Latremoliere13}.

The reach informs us on how far the sets:
\begin{equation*}
\left\{\varphi\circ\pi_\A:\varphi\in\StateSpace_1(\omega)\right\}\text{ and }\left\{\varphi\circ\pi_\B:\varphi\in\StateSpace_1(\omega)\right\}
\end{equation*}
are for the Hausdorff distance defined on $\StateSpace(\D)$ by $\Kantorovich{\Lip_\D}$. To be of use, we also must measure how far these two sets are from the whole state space. This leads us to:
\begin{definition}\label{height-def}
The \emph{height} $\bridgeheight{\gamma}{\Lip_1,\Lip_2}$ of a bridge $\gamma = (\D,\pi_1,\pi_2,\omega)$ from a {\Lqcms} $(\A_!,\Lip_1)$ to a {\Lqcms} $(\A_2,\Lip_2)$ is the nonnegative real number:
\begin{equation*}
\max\left\{ \Haus{\Kantorovich{\Lip_j}}\left(\StateSpace(\A_j),\{\mu\circ\pi_j:\mu\in\StateSpace_1(\omega)\}\right) : j\in\{1,2\} \right\}\text{.}
\end{equation*}
\end{definition}

The length of a bridge is thus naturally defined as:
\begin{definition}\label{length-def}
The length $\bridgelength{\gamma}{\Lip_\A,\Lip_\B}$ of a bridge $\gamma$ from a {\Lqcms} $(\A,\Lip_\A)$ to a {\Lqcms} $(\B,\Lip_\B)$ is defined as:
\begin{equation*}
\max\left\{\bridgeheight{\gamma}{\Lip_\A,\Lip_\B},\bridgereach{\gamma}{\Lip_\A,\Lip_\B} \right\}\text{.}
\end{equation*}
\end{definition}

One may check that, for any $\lambda > 0$ such that $\lambda\geq\bridgelength{\gamma}{\Lip_\A,\Lip_\B}$, the seminorm $\Lip_{\gamma,\lambda}$ is indeed an admissible Leibniz Lip-norm \cite[Theorem 6.3]{Latremoliere13}.

Though it would be natural to define the quantum propinquity between two {\Lqcms s} as the infimum of the lengths of all possible bridges between them, such a construction would fail to satisfy the triangle inequality. Instead, we propose:
\begin{definition}\label{propinquity-def}
Let $(\A,\Lip_\A)$ and $(\B,\Lip_\B)$ be two {\Lqcms s}. The quantum Gromov-Hausdorff propinquity $\propinquity((\A,\Lip_\A),(\B,\Lip_\B))$ is given as:
\begin{equation*}
\inf\left\{ \sum_{j=1}^n \bridgelength{\gamma_j}{\Lip_j,\Lip_{j+1}}\middle\vert
\begin{array}{l}
n\in\N,n>0,\\
\forall j\in\{1,\ldots,n+1\}\quad(\A_j,\Lip_j)\text{ is a }\\\hspace{1cm}\text{{\Lqcms}},\\
\forall j\in\{1,\ldots,n\}\quad \gamma_j \text{ is a bridge from $\A_j$ to $\A_{j+1}$,}\\
(\A_1,\Lip_1)=(\A,\Lip_\A)\text{ and }(\A_{n+1},\Lip_{n+1})=(\B,\Lip_\B)
\end{array}
\right\}\text{.}
\end{equation*}
\end{definition}

The main results of \cite{Latremoliere13} now read:
\begin{theorem}[Theorem 5.13, Theorem 6.1,\cite{Latremoliere13}]
Let $(\A,\Lip_\A)$, $(\B,\Lip)$ and $(\D,\Lip_\D)$ be three {\Lqcms s}. Then:
\begin{enumerate}
\item $\propinquity((\A,\Lip_\A),(\B,\Lip_\B)) = \propinquity((\B,\Lip_\B),(\A,\Lip_\A))$,
\item $\propinquity((\A,\Lip_\A),(\D,\Lip_\B)) \leq \propinquity((\A,\Lip_\A),(\B,\Lip_\B)) + \propinquity((\B,\Lip_\B),(\D,\Lip_\B))$,
\item $\propinquity((\A,\Lip_\A),(\B,\Lip_\B)) = 0$ if and only if there exists a *-isomorphism $h : \A\rightarrow\B$ such that $\Lip_\B\circ h = \Lip_\A$.
\end{enumerate}
\end{theorem}

\begin{theorem}[Theorem 6.3, Corollary 6.4, Theorem 6.6 ,\cite{Latremoliere13}]
The quantum propinquity dominates the quantum Gromov-Hausdorff distance. Moreover, for any two compact metric spaces $(X,\mathsf{m}_X)$ and $(Y,\mathsf{m}_Y)$, if $\Lip_X$ and $\Lip_Y$ are the respective Lipschitz seminorms on $C(X)$ and $C(Y)$ for $\mathsf{m}_X$ and $\mathsf{m}_Y$ (see Equation (\ref{Lip-eq}), then $\propinquity((C(X),\Lip_X),(C(Y),\Lip_Y))$ is less or equal to the Gromov-Hausdorff distance between $(X,\mathsf{m})$ and $(Y,\mathsf{m}_Y)$.
\end{theorem}

Thus, as desired, the quantum propinquity defines a metric on the class of {\Lqcms s}, up to isometry and *-isomorphism, which extends the notion of Gromov-Hausdorff convergence to noncommutative metric spaces and dominates Rieffel's quantum Gromov-Hausdorff distance while providing a mean to work exclusively within the category of {\Lqcms s}.

Our quantum propinquity is dominated by the unital nuclear distance introduced by Kerr and Li in \cite{kerr09}. Thus, continuity results, including regarding quantum tori and fuzzy tori, which are proven valid for the unital nuclear distance, are also valid for our metric. However, the use of a subtrivialization argument \cite{Blanchard97} and Hahn-Banach theorem in the proof that fuzzy tori approximate quantum tori in the unital nuclear distance \cite{li03,kerr09} does not provide a very natural bridge, in the sense of Definition (\ref{bridge-def}). The quantum propinquity, as we shall prove in this paper, provides a much more natural construction of a bridge, and thus by Equation (\ref{admissible-eq}), of much more natural admissible Leibniz Lip-norms, to compute estimates on the distances between quantum and fuzzy tori.

In \cite{Latremoliere13b}, motivated by the question of completeness of the quantum propinquity, we introduce a ``dual'' metric between {\Lqcms s} which we call the dual Gromov-Hausdorff propinquity $\mathsf{\Lambda}^\ast$. This new metric is shown to be complete, and is dominated by the quantum propinquity. In this paper, we shall prove a convergence result for the quantum propinquity, and thus our result carries to the dual Gromov-Hausdorff propinquity as well.

Last, it should be noted that in \cite{Latremoliere13,Latremoliere13b}, we allow for various modifications of the quantum and dual propinquities to fit various setups, by restricting the types of bridges one may employ in the computation of the propinquity between two {\Lqcms s}. In this paper, we work with the default setting described in this section. However, one could use the modified versions introduced in \cite{Latremoliere13,Latremoliere13b} restricted to C*-compact metric spaces \cite{Rieffel10c} and apply the following argument unchanged to these stronger propinquity metrics.

\section{Quantum and Fuzzy Tori}

Quantum tori have a long and rich history as the foundation for the field of noncommutative geometry. Their origin may be traced to the notion of C*-crossed products introduced by Zeller-Meier \cite{Zeller-Meier68}, when applied to the action of $\Z$ on the circle $\T$ generated by a rotation of angle $2\pi\theta$ ($\theta \in [0,1)$), leading to the rotation algebra $\A_\theta$. Quantum tori are defined as twisted group C*-algebras of $\Z^d$ for any $d>1$, with the rotation algebras giving all the quantum tori for $d=2$. Following on the work of Effr{\"o}s and Hahn in \cite{EffrosHahn67} about transformation group C*-algebras, Rieffel's surprisingly constructed non-trivial projections in $\A_\theta$ for any irrational $\theta$. The matter of classifying projections in these rotation algebras was linked to the classification of  the rotation algebras themselves, and led to an important chapter in the development of $K$-theory for C*-algebra theory \cite{PimVoi80a,Pimsner80}. Connes proposed in \cite{Connes80} to consider rotation algebras as noncommutative generalizations of differentiable manifolds, by exploiting the transport of structure from the torus to rotation algebras made possible via the dual action. This early, ground-breaking work open the field of noncommutative geometry, where quantum tori remain to this day a central class of examples. A very incomplete list of examples of such developments can be found in the work of Rieffel on noncommutative vector bundles over quantum tori \cite{Rieffel83,Rieffel87b,Rieffel88,Rieffel87c}), noncommutative Yang-Mills theory \cite{ConnesRieffel87,Rieffel90a}, and the summary presented in \cite{Rieffel90}, to mention but a few.

Noncommutative tori are not approximately finite (AF), since their $K_1$ groups are nontrivial --- though Pimsner and Voiculescu showed in \cite{PimVoi80a} how to embed irrational rotation algebras in AF algebras. Irrational rotation algebras, however, are inductive limits of direct sum of circle algebras (AT), as shown by Elliott and Evans \cite{Elliott93b}, which started the program of classification for AT algebras and further developments in classification theory for C*-algebras. Quantum tori also provide a fundamental model for strict quantization \cite{Rieffel90}, i.e. can be seen as a continuous deformation of the commutative C*-algebra $C(\T^d)$. 

More recently, interests in models of physical interests on ``fuzzy spaces'', i.e. spaces described by matrix algebras \cite{Madore}, grew within the mathematical physics community, with particular emphasis on models developed on ``fuzzy tori'', e.g. in \cite{tHooft02,Connes97,Zumino98}. A common problem in these models is the study of the behavior of various quantities when the dimension of the underlying matrix algebras grow to infinity, with the hope that, in some sense, one may obtain in the limit a noncommutative physics theory on quantum tori. Yet, as we noted above, quantum tori are not AF, and thus the question of what an adequate limiting process for these physical constructions should be was left unanswered for some time.

Motivated by this question, Rieffel proposed the quantum Gro\-mov-Haus\-dorff distance \cite{Rieffel00} as a new mean to define approximations of C*-algebras endowed with metric information. We then proved in \cite{Latremoliere05} that, indeed, quantum tori are limits of fuzzy tori for this new form of convergence. As we have discussed, the quantum Gro\-mov-Haus\-dorff distance is somewhat unaware of the product structure of the underlying quantum spaces, and the need for a stronger version of this metric arose in recent developments in noncommutative metric geometry. This is our motivation to introduce the quantum propinquity, as explained in the introduction. Now, we must address the challenge of convergence of matrix algebras to quantum tori for our new metric once again. One approach is to use the domination of the quantum propinquity by the unital nuclear distance \cite{kerr09}, but this does not lead to explicit bridges and is not very amenable to quantitative arguments. We thus are led to seek a more quantitative and direct approach. This is the matter addressed in this paper.

\subsection{Background}

This preliminary subsection contains a brief summary of the various facts and notations we will use in our work with quantum and fuzzy tori. 

\begin{notation}
Let $\N_\ast = \N \setminus \{0,1\}$.
\end{notation}

\begin{notation}
Let $\Nbar = \N_\ast\cup\{\infty\}$ be the one-point compactification of $\N_\ast$. For any $d \in \N_\ast$ and $k = (k_1,k_2,\ldots,k_d) \in \Nbar^d$, we set:
\begin{equation*}
k\Z^d = \prod_{j=1}^d k_j \Z \quad\text{ and }\quad\Z^d_k = \Mod{\Z^d}{k\Z^d} \text{,}
\end{equation*} 
with the convention that $\infty\Z = \{0\}$, so that $\Z^d_{(\infty,\ldots,\infty)} = \Z^d$. The Pontryagin dual of $\Z^d_k$ is denoted by $\U^d_k$. In particular, if $k\in \N^d$ then $\Z^d_k$ is finite and thus self-dual. However, we shall always consider $\U_k^d$ as a compact subgroup of the $d$-torus $\U^d = \U^d_{(\infty,\ldots,\infty)}$, where $\U = \{ z \in \C : |z| = 1 \}$ is the unitary group of $\C$.
\end{notation}

\begin{notation}
Let $d\in \N_\ast$. The set of all skew-bicharacters of $\Z^d_k$ is denoted by $\mathds{B}_k^d$. We define:
\begin{equation*}
\Xi^d = \left\{ (k,\sigma) : k \in \Nbar^d\text{ and } \sigma \in \mathds{B}_k^d \right\}\text{.}
\end{equation*} 
We identify $\Xi^d$ with the subset of $\Nbar^d \times \mathds{B}^d_{(\infty,\ldots,\infty)}$ consisting of pairs $(k,\sigma)$ such that $\sigma$ is the unique lift of an element of $\mathds{B}_k^d$ to $\mathds{B}^d_{(\infty,\ldots,\infty)}$. With this identification, $\Xi^d$ is topologized as a topological subspace of $\Nbar^d \times \mathds{B}^d_{(\infty,\ldots,\infty)}$ endowed with the product topology, where $\mathds{B}^d_{(\infty,\ldots,\infty)}$ is endowed with the topology of pointwise convergence.
\end{notation}

We can easily identify $\mathds{B}^d_{(\infty,\ldots,\infty)}$, endowed with the topology of pointwise convergence, with the quotient of the space $\Theta$ of all skew-bilinear forms from $\Z^d$ to $\R$ by the equivalence relation defined by the kernel of the map:
\begin{equation*}
\theta \in \Theta \longmapsto \left( x,y \in \Z^d \mapsto \exp\left(i\pi\theta(x,y)\right) \right)
\end{equation*}
where $\Theta$ is endowed with its usual topology as a subset of $d\times d$ matrices over $\R$. In particular, $\B^1_\infty$ can be identified with $\U$, or equivalently, with $\Mod{\R}{\Z}$. More generally, $\Nbar^d\times \mathds{B}^d_{(\infty,\ldots,\infty)}$ is a compact set, and as $\Xi^d$ is easily checked to be closed, it is a compact set as well.

\bigskip

For any (nonempty) set $E$ and any $p \in [1,\infty)$, the set $\ell^p(E)$ is the set of all $p$-summable complex valued families over $E$, endowed with the norm:
\begin{equation*}
\|\xi\|_p = \left(\sum_{x \in E} |\xi_x|^p\right)^{\frac{1}{p}}
\end{equation*}
for all $\xi = (\xi_x)_{x\in E} \in \ell^p(E)$.

The main objects of this paper are quantum and fuzzy tori, which are the enveloping C*-algebras of *-algebras defined by twisting the convolution products on finite products of cyclic groups, as follows. We refer to \cite{Zeller-Meier68} for a general reference on twisted convolution C*-algebras of discrete groups. In general, one employs $\U$-valued multipliers of $\Z^d_k$ to twist the convolution on $\ell^1(\Z^d)$, for any $d\in\N_\ast$ and $k\in\Nbar^d$. However, two cohomologous multipliers will lead to isomorphic algebras through such a construction, and by \cite{Kleppner65}, any $\U$-valued multiplier of $\Z^d_k$ is cohomologous to a skew-bicharacter of $\Z^d_k$. This justifies that we restrict our attention to deformation given by skew-bicharacters in the rest of this paper. We now can define:

\begin{definition}
Let $d\in \N_\ast$, $k \in \Nbar^d$ and $\sigma$ a skew bicharacter of $\Z^d_k$. The \emph{twisted convolution product} on $\ell^1\left(\Z^d_k\right)$ of any $f,g \in \ell^1\left(\Z^d_k\right)$ is:
\begin{equation*}
f \conv{k,\sigma} g : n \in \Z^d_k \longmapsto \sum_{m\in \Z^d_k} f(m)g(n-m)\sigma(m,n)\text{.}
\end{equation*}
\end{definition}

For any $d\in\N_\ast$, $k\in\Nbar^d$ and any skew-bicharacter $\sigma$ of $\Z^d_k$, the $2$-cocycle identity satisfied by $\sigma$ implies $\left(\ell^1\left(\Z^d_k\right),\conv{k,\sigma}\right)$ is indeed an associative algebra; moreover, this algebra carries a natural involution:

\begin{definition}
Let $d\in \N_\ast$, $k\in \Nbar^d$. For any $f \in \ell^1\left(\Z^d_k\right)$, we define the \emph{adjoint} $f^\ast$ of $f$ by:
\begin{equation*}
f^\ast : n \in \Z^d_k \longmapsto \overline{f(-n)} \text{.}
\end{equation*}
\end{definition}
We note that the adjoint operation does not depend on the choice of a skew bicharacter of $\Z^d_k$.

For any $d\in \N_\ast$ and $(k,\sigma)\in\Xi^d$, one checks easily that $\left(\ell^1\left(\Z^d_k\right),\conv{\sigma},\cdot^\ast\right)$ is a *-algebra , and we now wish to construct its enveloping C*-algebra. To do so, we shall choose a natural faithful *-representation of $\left(\ell^1\left(\Z^d_k\right),\conv{\sigma},\cdot^\ast\right)$ on $\ell^2\left(\Z^d_k\right)$.

For any set $E$, the Banach space $\ell^2(E)$ is a Hilbert space for the inner product:
\begin{equation*}
\xi,\eta \in \ell^2(E) \longmapsto \inner{\xi}{\eta} = \sum_{x\in E} \overline{\xi_x}\eta_x\text{.}
\end{equation*}
The canonical Hilbert basis $(e_x)_{x\in E}$ of $\ell^2(E)$ is given by:
\begin{equation*}
e_x : y \in E\mapsto \begin{cases} 1 \text{ if $y=x$}\\ 0 \text{ otherwise.}\end{cases}
\end{equation*}

As we shall work with representations on $\ell^2(\Z^d)$ in this whole paper, the following notation will prove useful.

\begin{notation}
The concrete C*-algebra of all bounded linear operators on $\ell^2(\Z^d)$ will be simply denoted by $\B^d$. We shall denote the norm for bounded linear operators on $\ell^2(\Z^d)$ by $\|\cdot\|_{\B^d}$.
\end{notation}

With these notations set, the chosen representation of $\left(\ell^1\left(\Z^d_k\right),\ast_\sigma,\cdot^\ast\right)$ used to construct our C*-algebras is given by:

\begin{theorem}\label{representation-thm}
Let $d\in \N_\ast$, $k \in \Nbar^d$ and $\sigma$ be a skew bicharacter of $\Z^d_k$. We define the operator $U_{k,\sigma}^n$ as the unique bounded linear operator of $\ell^2\left(\Z^d_k\right)$ such that:
\begin{equation}\label{U-eq}
U_{k,\sigma}^n e_m = \sigma(m,n)e_{m-n}
\end{equation}
for all $m\in \Z^d_k$, where $(e_m)_{m\in\Z^d_k}$ is the canonical Hilbert basis of $\ell^2\left(\Z^d_k\right)$.

The map:
\begin{equation*}
\rho_{k,\sigma} : f \in \ell^1\left(\Z^d_k\right) \longmapsto \sum_{n\in\Z^d_k} f(n)U_{k,\sigma}^n
\end{equation*} 
is a faithful *-representation of the *-algebra $\left(\ell^1\left(\Z^d_k\right),\conv{k,\sigma},\cdot^\ast\right)$ on $\ell^2\left(\Z^d_k\right)$. 
\end{theorem}

\begin{proof}
For any $n,m \in \Z^d_k$, we have:
\begin{equation*}
\begin{split}
U_{k,\sigma}^n U_{k,\sigma}^m e_j &= U_{k,\sigma}^n \sigma(j,m)e_{j-m} \\
&= \sigma(j-m,n)\sigma(j,m)e_{j-m-n} \\
&= \overline{\sigma(m,n)} \sigma(j,n)\sigma(j,m)e_{j-m-n} = \sigma(n,m)\sigma(j,n+m)e_{j-(m+n)}\\
&= \sigma(n,m) U_{k,\sigma}^{n+m} e_j \text{.}
\end{split}
\end{equation*}

Thus, $n \in \Z^d_k \mapsto U_{k,\sigma}^n$ is a $\sigma$-projective unitary representation of $\Z^d_k$ on $\ell^2\left(\Z^d_k\right)$. Naturally, its integrated representation defines a *-representation for the *-algebra $(\ell^1(\Z^d),\conv{k,\sigma},\cdot^\ast)$; moreover it is faithful in this case: this is a standard result \cite{Zeller-Meier68}.
\end{proof}

\begin{definition}
The C*-algebra $C^\ast\left(\Z^d_k,\sigma\right)$ is the norm closure of $\rho_{k,\sigma}\left(\ell^1\left(\Z^d_k\right)\right)$. 
\end{definition}

\begin{notation}
The norm of $C^\ast\left(\Z^d_k,\sigma\right)$ is denoted by $\|\cdot\|_{k,\sigma}$. In proofs, we will often denote $C^\ast\left(\Z^d_k,\sigma\right)$ simply by $\A_{k,\sigma}$.
\end{notation}

 As a matter of terminology, the C*-algebras $C^\ast\left(\Z^d,\sigma\right)$ are called \emph{quantum tori} for any skew bicharacter $\sigma$ of $\Z^d$. For $k \in \N_\ast^d$, the C*-algebra $C^\ast\left(\Z^d_k,\sigma\right)$ is called a \emph{fuzzy torus}, and are finite sums of matrix algebras.

\begin{notation}
We shall identify $f \in \ell^1\left(\Z^d_k\right)$ with $\rho_{k,\sigma}(f)$ whenever no confusion may arise, for any $(k,\sigma)\in\Xi^d$, since $\rho_{k,\sigma}$ is faithful.
\end{notation}

\begin{remark}\label{full-rmk}
Technically, we have defined the \emph{reduced} twisted C*-algebras of $\Z^d_k$, but since our groups are Abelian, they are amenable, and thus our definition also coincides with the full twisted C*-algebras of $\Z^d_k$. To be somewhat more precise, we recall \cite{Davidson} that the full C*-algebra of $\Z^d_k$ twisted by a skew bicharacter $\sigma$ is the completion of $\ell^1\left(\Z^d_k,\ast_{k,\sigma},\cdot^{\ast}\right)$ for the C*-norm defined by:
\begin{equation*}
\|f\|_{\mathrm{full},k,\sigma} = \sup \left\{ \|\pi(f)\|_{\B(\Hilbert)} : \begin{array}{l}\text{$\Hilbert$ is a Hilbert space and,}\\\text{$\pi$ is a non-degenerate *-representation}\\\text{of $\ell^1\left(\Z^d_k,\ast_{k,\sigma},\cdot^{\ast}\right)$ on $\Hilbert$}\end{array} \right\}\text{,}
\end{equation*}
for all $f\in\ell^1\left(\Z^d_k\right)$. One note that we have proven that there exists at least one *-representation of $\ell^1\left(\Z^d_k,\ast_{k,\sigma},\cdot^{\ast}\right)$, given by $\rho_{k,\sigma}$, and that the norm of $\pi(f)$ is dominated by $\|f\|_1$ for all $f\in\ell^1\left(\Z^d_k\right)$. As $\rho_{k,\sigma}$ in particular is faithful, the C*-seminorm $\|\cdot\|_{\mathrm{full},k,\sigma}$ thus constructed in a norm. The full C*-algebra $C_{\mathrm{full}}^\ast(\Z^d_k,\sigma)$ is the completion of $\ell^1\left(\Z^d_k,\ast_{k,\sigma},\cdot^{\ast}\right)$ for $\|\cdot\|_{\mathrm{full},k,\sigma}$ and is universal by construction for all unitary $\sigma$-representations of $\Z^d_k$. Remarkably, one can prove that, in fact,  $C^\ast\left(\Z^d_k,\sigma\right) = C^\ast_{\mathrm{full}}\left(\Z^d_k,\sigma\right)$ as $\Z^d_k$ is amenable since Abelian \cite{Davidson}. 
\end{remark}

Quantum and fuzzy tori carry natural group actions, from which we shall derive their metric geometry from a kind of transport of structure. For all $\omega = (\omega_1,\ldots,\omega_d) \in \U^d_k$ and $n = (n_1\ldots,n_d) \in \Z^d_k$, we shall denote $\left(\omega_1^{n_1},\ldots,\omega_d^{n_d}\right)$ by $\omega^n$, so that $(\omega,n)\in\U^d_k\times \Z^d_k\mapsto \omega^n$ is the Pontryagin duality pairing. We quote the following two results which define the quantum metric space structure of quantum and fuzzy tori as quantum homogeneous spaces of the torus and its compact subgroups. 

\begin{theorem-definition}\cite{Zeller-Meier68}\label{dual-action-thmdef}
Let $d \in \N_\ast$, $k\in \Nbar^d$ and $\sigma$ a skew-bicharacter of $\Z^d_k$. For all $\omega \in \U^d_k$, there exists a unique *-automorphism $\alpha_{k,\sigma}^\omega$ of $C^\ast\left(\Z^d_k,\sigma\right)$ such that for any $f \in \ell^1(\Z^d_k)$ and any $n \in \Z^d_k$, we have:
\begin{equation*}
\alpha_{k,\sigma}^\omega(f)(n) = \omega^n f(n)\text{.}
\end{equation*}
The map $\omega\in\U^d_k\mapsto \alpha_{k,\sigma}^\omega$ is a strongly continuous action of $\U^d_k$ on $C^\ast\left(\Z^d_k,\sigma\right)$. This action is called the \emph{dual action} of $\U^d_k$ on $C^\ast\left(\Z^d_k,\sigma\right)$.
\end{theorem-definition}

We note that Theorem (\ref{dual-action-thmdef}) follows in part from the universal property of quantum tori and fuzzy tori discussed briefly in Remark (\ref{full-rmk}).

\subsection{Metric Structures on Quantum and Fuzzy Tori}

We now describe how Rieffel endowed quantum and fuzzy tori, and more general ``quantum homogeneous spaces'', with a metric structure, in \cite{Rieffel98a}.

A \emph{length function} $l$ on a group $G$ is a function from $G$ to the nonnegative real numbers such that:
\begin{enumerate}
\item for all $x,y \in G$ we have $l(xy)\leq l(x) + l(y)$,
\item for all $x\in G$, we have $l(x) = 0$ if and only if $x$ is the neutral element in $G$,
\item for all $x\in G$, we have $l\left(x^{-1}\right) = l(x)$.
\end{enumerate}
A length function allows the definition of a left invariant metric on $G$ by setting $x,y\in G \mapsto l(x^{-1}y)$, and conversely a left invariant metric $\mathrm{d}$ on $G$ defines a length function $x\in G\mapsto \mathrm{d}(x,e)$ with $e$ the neutral element of $G$. If $G$ is a compact group with Haar probability measure $\lambda$ and $\mathrm{d}$ is a metric on $G$ which induces its topology, then $x\in G\mapsto \int_G d(gx,g)\,d\lambda(g)$ is a continuous length function on $G$. If moreover, $\tau$ is a compact topology on a group $G$ and $l$ is a continuous length function on $G$ for $\tau$, then the topology generated on $G$ by the distance induced by $l$, being Haus\-dorff and coarser than $\tau$ which is compact, is in fact the topology $\tau$. Thus, every compact metrizable group admits a length function whose distance metrizes the given topology. With this in mind, we have:

\begin{theorem-definition}\cite{Rieffel98a}
Let $d\in \N_\ast$, $k\in \Nbar^d$, $\sigma$ a skew-bicharacter of $\Z^d_k$ and $l$ be a continuous length function on $\U^d_k$. Let $\alpha_{k,\sigma}$ be the dual action of $\U^d_k$ on $C^\ast\left(\Z^d_k,\sigma\right)$. For all $a\in C^\ast\left(\Z^d_k,\sigma\right)$, we set:
\begin{equation*}
\Lip_{l,k,\sigma} (a) = \sup \left\{ \frac{\|a-\alpha_{k,\sigma}^\omega(a)\|_{k,\sigma}}{l(\omega)} : \omega \in \U^d_k \setminus \{1\} \right\} \text{,}
\end{equation*}
where $1$ is the unit of $\U^d_k$. Then $\left(C^\ast\left(\Z^d_k,\sigma\right),\Lip_{l,k,\sigma}\right)$ is a {\Lqcms}.
\end{theorem-definition}

\begin{remark}
Any subgroup of $\Z^d$ is isomorphic to a group $k\Z^d$ for some $k\in \N_\ast^d$, via an invertible $\Z$-module map of $\Z^d$. By functoriality of twisted C*-algebras of groups, we see that we can always adjust the situation so that, whatever subgroups we may consider in $\Z^d$, our current work applies, up to obvious changes to the metric. Our current setup, however, makes the computations more approachable.
\end{remark}

\subsection{A first approximation result}

We conclude this section with a fundamental approximation result. This result was proven in \cite{Latremoliere05}, and it will be quite useful in our current work, so we present briefly the general scheme of its proof. We start with the following definition:

\begin{definition}
Let $d\in\N_\ast$, $k\in\Nbar^d$ and $\sigma$ a skew bicharacter of $\Z^d_k$. For any continuous function $\phi : \U^d_k \rightarrow \R$, we set:
\begin{equation*}
\alpha_{k,\sigma}^\phi : a \in C^\ast\left(\Z^d_k,\sigma\right) \mapsto \int_{\U^d_k} \phi(\omega)\alpha_{k,\sigma}^\omega(a)\,d\lambda_k(\omega)\text{,}\
\end{equation*}
where $\lambda_k$ is the Haar probability measure on $\U^d_k$.
\end{definition}

\begin{theorem}\label{Fejer-approx-thm}
Let $d\in\N_\ast$ and $k\in \Nbar^d$. Let $l$ be a continuous length function on $\U^d_k$. Let $\varepsilon > 0$. There exists a positive, continuous function $\phi : \U^d \rightarrow \R$ and a neighborhood $U$ of $k$ in $\Nbar^d$ such that:
\begin{enumerate}
\item For all $c \in U$, all skew-bicharacter $\sigma$ of $\Z^d_c$ and for all $a\in \sa{\C^\ast(\Z^d_c,\sigma)}$ we have $\|a-\alpha_{c,\sigma}^\phi (a)\|_{c,\sigma} \leq \varepsilon \Lip_{l,c,\sigma}(a)$ and $\Lip_{l,c,\sigma}(\alpha_{c,\sigma}^\phi(a))\leq\Lip_{l,c,\sigma}(a)$,
\item There exists a finite subset $S$ of $\Z^d$ with $0\in S$ such that, for all $c \in U$ and any skew-bicharacter $\sigma$ of $\Z^d_c$, the restriction of the canonical surjection $q_c: \Z^d \rightarrow \Z^d_c$ is injective on $S$ and the range of $\alpha_{c,\sigma}^\phi$ is the span of $\{ U_{c,\sigma}^p : p \in q_c(S) \}$, where the unitaries $U_{c,\sigma}^p$ ($p\in\Z^d_c$) are defined by Equality (\ref{U-eq}) in Theorem (\ref{representation-thm}).
\end{enumerate}
\end{theorem}

Due to the importance of this result, we provide a sketch of its proof based upon the following three lemmas, proven in greater generality in \cite{Latremoliere05}. 

\begin{lemma}\label{Fejer-lemma}
Let $d\in \N_\ast$ and $k\in \Nbar^d$. Let $f : \U^d_k \rightarrow \C$ be a continuous function such that $f(1,\ldots,1) = 0$. Let $\varepsilon > 0$. There exists a finite linear combination $\phi$ of characters of $\U^d_k$, including the trivial character, such that $\phi \geq 0$, $\int_{\U^d_k} \phi\,d\lambda_k = 1$ and $\int_{\U^d_k} \phi |f| \,d\lambda_k \leq \varepsilon$.
\end{lemma}

\begin{proof}
See \cite[Lemma 3.1]{Latremoliere05}, where the compact group $G$ is chosen to be $\U^d_k$.
\end{proof}

\begin{lemma}\label{range-lemma}
Let $d\in\N_\ast$ and $(k,\sigma)\in\Xi^d$. If $\phi$ is a linear combination of characters of $\U^d_k$, then the range of $\alpha_{k,\sigma}^\phi$ is the span of $\{ U^p_{k,\sigma} : \widehat{\phi}(p) \not=0 \}$, where $\widehat{\phi}$ is the Fourier transform of $\phi$. In particular, $\alpha_{k,\sigma}^\phi$ has finite rank.
\end{lemma}

\begin{proof}
See \cite[Lemma 3.2]{Latremoliere05}, with $\Sigma = \U^d_k$.
\end{proof}

\begin{lemma}\label{measure-cv-lemma}
Let $d\in \N_\ast$ and $k\in \Nbar^d$. If $f : \U^d_k \rightarrow \C$ is a continuous function, then:
\begin{equation*}
\lim_{c\rightarrow k} \int_{\U^d_c} f\,d\lambda_c = \int_{\U^d_k} f\,d\lambda_k\text{,}
\end{equation*}
where $\lambda_c$ is meant for the probability Haar measure on $\U^d_c$ for all $c\in \Nbar^d$.
\end{lemma}

\begin{proof}
This is \cite[Lemma 3.6]{Latremoliere05} in the case of $G = \U_k$.
\end{proof}

\begin{proof}[Sketch of proof of Theorem (\ref{Fejer-approx-thm})]
Let $\varepsilon > 0$. By Lemma (\ref{Fejer-lemma}), there exists a finite linear combination $\phi$ of characters of $\U^d_k$ such that $\phi\geq 0$, $\int_{\U^d_k}\phi \,d\lambda_k = 1$ and:
\begin{equation*}
\int_{\U^d_k}\phi(\omega)l(\omega)\,d\lambda_k(\omega) < \frac{1}{2}\varepsilon\text{.}
\end{equation*}
Let $S = \{ p \in \Z_k^d : \widehat{\phi} (p) \not= 0\}$, where $\widehat{\phi} : \Z_k^d \rightarrow\C$ is the Fourier transform on $\phi$. Note that $0 \in S$ and $S$ is finite.

By Lemma (\ref{measure-cv-lemma}), there exists a neighborhood $U_0$ of $k$ in $\Nbar^d$ such that:
\begin{equation*}
\int_{\U^d_c} \phi(\omega)l(\omega) \,d\lambda_c(\omega) \leq \varepsilon\text{.}
\end{equation*}

For any $c =(c_1,\ldots,c_d),k=(k_1,\ldots,k_d)\in \Nbar^d$, we set $c\mid k$ when $c_j$ divides $k_j$ for all $j=1,\ldots,d$, with the convention that all $a\in \Nbar$ divide $\infty$. With these notations in mind, we define:
\begin{equation*}
V_0 = \left\{ c \in \Nbar^d : c|k \right\}\text{.}
\end{equation*}
For each $c \in V_0$, the group $\Z_c^d$ is a quotient group of $\Z_k^d$. Letting $q_c$ be the canonical surjection from $\Z^d_k$ onto $\Z^d_c$, we define:
\begin{equation*}
V = \left\{ c \in V_0 : \text{$q_c$ is injective on $S$} \right\}\text{.}
\end{equation*}
The set $V$ is easily checked to be open in $\Nbar^d$, and it contains $k$, so the set $U = U_0\cap V$ is an open neighborhood of $k$ in $\Nbar^d$.

Let $c\in U$ and $\sigma$ be a skew bicharacter of $\Z^d_c$. Since $c \in V$ in particular, $\U_c^d\subseteq \U_k^d$, and thus the restriction of $\phi$ to $\U_c^d$ is still a linear combination of characters of $\U_c^d$. Moreover, again by definition of $V$, the group $\Z^d_c$ is a quotient of $\Z^d_k$ and the canonical surjection is injective on $S$. By Lemma (\ref{range-lemma}), the operator $\alpha_{c,\sigma}^\phi$ has range the span of $\{ U^p_{c,\sigma} : p \in S\}$, and since $\phi$ is positive, the linear map $\alpha^\phi_{c,\sigma}$ is positive. Let $a\in \A_{c,\sigma}$.
\begin{equation*}
\begin{split}
\|a-\alpha_{c,\sigma}^\phi(a)\|_{c,\sigma} & \leq \int_{\U^d_c} \|a-\alpha_{c,\sigma}^\omega(a)\|_{c,\sigma}\phi(\omega)\,d\lambda_c(\omega)\\
&= \int_{\U^d_c} \frac{\|a-\alpha_{c,\sigma}^\omega(a)\|_{c,\sigma}}{l(\omega)}l(\omega)\phi(\omega)\,d\lambda_c(\omega)\\
&\leq \Lip_{l,c,\sigma}(a) \int_{\U^d_c}\phi(\omega)l(\omega)\,d\lambda_c(\omega)\\
&\leq \varepsilon \Lip_{l,c,\sigma}(a)\text{.}
\end{split}
\end{equation*}
Moreover, again using the fact that $c\in V$, we note that if $\widehat{\phi}_c$ is the Fourier transform of the restriction of $\phi$ to $\U_c^d$, then we have $\widehat{\phi}_c \circ q_c = \widehat{\phi}_k$ where $q_c  :\Z_k^d \rightarrow \Z_c^d$ is the canonical surjection (this follows from the fact that $q_c$ is injective on the support $S$ of $\phi$). From this, we can conclude that:
\begin{equation*}
1 = \int_{\U_k^d} \phi\,d\lambda_k = \widehat{\phi}_k(0) = \widehat{\phi}_c(0) = \int_{\U_c^d}\phi\,d\lambda_c \text{.}
\end{equation*}

Using lower-semi-continuity of the Lip-norm $\Lip_{l,c,\theta}$ as well as its invariance under the dual action :
\begin{equation*}
\begin{split}
\Lip_{l,c,\sigma}(\alpha^\phi(a))&\leq \int_{\U^d_c} \phi(\omega)\Lip_{l,c,\sigma}(\alpha_{c,\theta}^\omega(a))\,d\lambda_c(\omega)\\
&=\Lip_{l,c,\sigma}(a)\int_{\U_c^d}\phi\,d\lambda_c = \Lip_{l,c,\theta}(a)\text{,}
\end{split}
\end{equation*}
which proves our theorem.
\end{proof}

\section{Continuous fields}

An important feature of the quantum and fuzzy tori is that they can be seen as fibers of a single continuous field of C*-algebras \cite[Corollary 2.9]{Latremoliere05}. For our purpose, we need a concrete construction of such a field, which goes beyond our groupoid-based construction in \cite{Latremoliere05}. In this section, for a fix $d\in \Nbar^d$, we introduce faithful non-degenerate *-representations on the Hilbert space $\ell^2(\Z^d)$ of the C*-algebras $\C^\ast\left(\Z^d_k,\sigma\right)$ for $(k,\sigma)\in\Xi^d$, in such a way that, informally, these representations are pointwise continuous for the strong operator topology. The precise statement will require some notations, to be found in this section. We also note that our efforts in this section will allow to define unital *-monomorphism from all the quantum tori and fuzzy tori into a single C*-algebra, namely the C*-algebra $\B^d$ of all bounded linear operators on $\ell^2\left(\Z^d\right)$, and these *-morphisms will be part of our bridge constructions when we work with the quantum propinquity later on.

\subsection{Representations on a fixed Hilbert Space}

For all $x\in\R$, the integers $\lfloor x\rfloor$ and $\lceil x \rceil$ are, respectively, the largest integer smaller than $x$ and the smallest integer larger than $x$. Moreover, we use the conventions $\lfloor\pm \infty\rfloor = \lceil\pm\infty\rceil = \pm\infty$, and $\pm\infty+n=n+\pm\infty=m(\pm\infty)=\pm \infty$ for all $n\in\R$ and $m\in (0,\infty)\in\R$.

We start with the following simple object:

\begin{notation}
Let $d \in \N_\ast$ and $k = (k_1,\ldots,k_d) \in \Nbar$. Let:
\begin{equation*}
I_k = \prod_{j=1}^d \left\{ \left\lfloor \frac{1-k_j}{2} \right\rfloor,\ldots,\left\lfloor \frac{k_j-1}{2} \right\rfloor\right\}\text{.}
\end{equation*}
\end{notation}

We observe that, by construction, the set:
\begin{equation}\label{partition-eq}
\mathscr{P}_k = \left\{ I_k + n : n \in k\Z^d \right\}
\end{equation}
is a partition of $\Z^d$. This is not the usual partition of $\Z^d$ in cosets of $k\Z^d$, but we will find it a bit more convenient (though one could, at the expense of worse notations later on, work with the standard partition of $\Z^d$ in cosets of $k\Z^d$).

Fix $d\in\N_\ast$ and $k \in \Nbar^d$. The canonical surjection $q_k : \Z^d \rightarrow \Z^d_k$ restricts to a bijection from $I_k$ onto $\Z^d_k$. We thus can define an isometric embedding $\vartheta_k : \ell^2\left(\Z^d_k\right)\rightarrow \ell^2(\Z^d)$ by setting for all $\xi \in \ell^2\left(\Z^d_k\right)$:
\begin{equation}\label{vartheta-embedding-def}
\vartheta_k(\xi) : n\in \Z^d \longmapsto \begin{cases}\xi(q_k(n)) \text{ if $n\in I_k$}\\0 \text{ otherwise.}\end{cases}
\end{equation}
Since $\vartheta_k$ is an isometry by construction, $\vartheta_k^\ast \vartheta_k$ is the identity of $\ell^2\left(\Z^d_k\right)$. Therefore, for all skew bicharacter $\sigma$ of $\Z^d_k$, the map $\vartheta_k \pi_{k,\sigma}(\cdot)\vartheta_k^\ast$ is a \emph{non-unital} *-representation of $C^\ast\left(\Z^d_k,\sigma\right)$ on $\ell^2(\Z^d)$. To construct a non-degenerate representation (or, equivalently, unital *-monomorphisms), we proceed as follows. Since $\mathscr{P}_k$, defined by Equation (\ref{partition-eq}), is a partition of $\Z^d$, we have the following decomposition of $\ell^2(\Z^d)$ in a Hilbert direct sum:
\begin{equation}\label{l2-decomposition-k-eq}
\ell^2(\Z^d) = \bigoplus_{n \in k\Z^d} \overline{\operatorname{span} \{ e_j : j \in I_k + n \}}
\end{equation}
with $(e_j)_{j\in\Z^d}$ the canonical basis of $\ell^2(\Z^d)$ given by $e_m(n)\in\{0,1\}$ and $e_m(n) = 1$ if and only if $n=m$, for all $m,n\in\Z^d$. Note that the range of $\vartheta_k$  is $\overline{\operatorname{span} \{ e_j : j \in I_k  \}}$.

For all $n \in k\Z^d$, let 
\begin{equation*}
u_n : \overline{\operatorname{span}\{e_j : j \in I_k\}} \longrightarrow \overline{\operatorname{span}\{e_j : j \in I_k + n\}}
\end{equation*}
be the unitary defined by extending linearly and continuously the map:
\begin{equation*}
e_j \in \{ e_m : m \in I_k \} \longmapsto e_{j+n}\text{.}
\end{equation*}

We now define:

\begin{notation}\label{rep}
Let $d\in \N_\ast$ and $k \in \Nbar^d$. Let $\sigma$ be a skew-bicharacter of $\Z^d_k$, and $\rho_{k,\sigma}$ the representation on $\left(\ell^1\left(\Z^d_k\right),\conv{k,\sigma},\cdot^\ast\right)$ on $\ell^2(\Z^d)$ defined by Theorem (\ref{representation-thm}). 

Let $\xi \in \ell^2(\Z^d)$, and write $\xi = \sum_{j \in k\Z^d} \xi_j$ with $\xi_j \in \overline{\operatorname*{span}\{e_m : m \in I_k + j \}}$. Such a decomposition is unique by Equation (\ref{l2-decomposition-k-eq}). Define for all $a\in C^\ast(\Z^d_k,\sigma)$:
\begin{equation}\label{rep-eq}
\pi_{k,\sigma}(a)\xi = \sum_{j\in k\Z^d} u_j\vartheta_k \rho_{k,\sigma}(a)\vartheta_k^\ast u_j^\ast \xi_j \text{,}
\end{equation}
which is well-defined since $\|u_n\vartheta_k \rho_{k,\sigma}(a)\vartheta_k^\ast u_j^\ast \xi_j\|_2\leq\|a\|_{k,\sigma} \|\xi_j\|_2$ for all $j\in k\Z^d$, and $\sum_{j\in kZ^d}\|\xi_j\|^2_2 = \|\xi\|_2^2 < \infty$ by definition of $\xi$.

It is easy to check that $\pi_{k,\sigma}$ thus defined is a faithful, non-degenerate (i.e. unital) *-representation of $C^\ast(\Z_k^d,\sigma)$ on $\ell^2(\Z^d)$, which acts ``diagonally'' in the decomposition of $\ell^2(\Z^d)$ given by Equation (\ref{l2-decomposition-k-eq}).
\end{notation}

\begin{remark}
We wish to emphasize that the only reason to work with the non-degenerate representations $\pi_{k,\sigma}$ instead of $\operatorname{Ad}_{\vartheta_{k,\sigma}}\circ\rho_{k,\sigma}$ ($(k,\sigma)\in\Xi^d$) is that they can be seen as \emph{unital} *-monomorphisms from the quantum and fuzzy tori into $\B^d$, which is an essential tool to build bridges as defined in Definition (\ref{bridge-def}).
\end{remark}

The representations introduced in Notation (\ref{rep}) enjoy an important topological property with respect to the strong operator topology of $\ell^2(\Z^d)$. To make this statement precise, we first choose a way to see an element of $\ell^1(\Z^d)$ as an element of $\ell^1\left(\Z^d_k\right)$ for any $k\in\Nbar^d$ as follows.

\begin{notation}\label{canonical-basis-notation}
For any $d\in \N_\ast$, $k\in\Nbar^d$, and $m\in\Z^d_k$, we define the element $\delta_m$ of $\ell^1(\Z_k^d)$ by:
\begin{equation*}
\delta_m : n \in \Z^d_k \longmapsto \begin{cases}
1 \text{ if $n=m$,}\\
0 \text{ otherwise.}
\end{cases}
\end{equation*}
Note that we defined the canonical basis $(e_n)_{n\in\Z^d}$ of $\ell^2(\Z^d)$ in such a way that $e_n=\delta_n$ for all $n\in\Z^d$, so one may find our present notation redundant. However, it will make our exposition much clearer if we distinguish between the vectors $e_n \in \ell^2(\Z^d)$ and the functions $\delta_n \in \ell^1(\Z^d)$ for all $n\in\Z^d$.
\end{notation}

\begin{lemma}\label{upsilon-lemma}
Let $d \in \N_\ast$ and $k \in \Nbar^d$. Let $q_k : \Z^d \rightarrow \Z^d_k$ be the canonical surjection. For any $f \in \ell^1(\Z^d)$ and any $n \in \Z^d$, we define:
\begin{equation*}
\upsilon'_k (f) (n) = \sum_{m \in k\Z^d} f( m + n ) \in \C \text{.}
\end{equation*}
Then $\upsilon_k'(f)(n) = \upsilon_k'(f)(m)$ for any $n,m\in\Z^d$ with $q_k(n)=q_k(m)$. Thus there exists a unique function $\upsilon_k(f) : \Z^d_k \rightarrow \C$ such that $\upsilon_k(f) \circ q_k = \upsilon_k'(f)$.

The map $\upsilon_k$ is a linear surjection of norm 1 such that $\upsilon_k(f^\ast) = \upsilon_k(f)^\ast$ for all $f\in\ell^1\left(\Z^d\right)$. Moreover, the restriction of $\upsilon_k$ to the functions in $\ell^1(\Z^d)$ supported on $I_k$ is injective.
\end{lemma}

\begin{proof}
Let $n,m \in \Z^d$ such that $q_k(m)=q_k(n)$, i.e. such that $n-m \in k\Z^d$. Thus $\sum_{v\in k\Z^d} f(n+v) = \sum_{v\in k\Z^d}f(m+((n-m)+v)) =\sum_{v\in k\Z^d}f(m+v)$ as desired. Hence $\upsilon_k$ is well-defined, and linearity is straightforward, as well as self-adjointness. Now, the function $\vartheta_k :\ell^2\left(\Z^d_k\right)\rightarrow\ell^2(\Z^d)$ defined in Equation (\ref{vartheta-embedding-def}) restricted to $\ell^1\left(\Z^d_k\right)$ is a function with values in $\ell^1(\Z^d)$, and we check that $\upsilon_k(\vartheta_k(g))=g$. 

Last, assume that we are given a function $f : \Z^d \rightarrow\C$  whose support is contained in $I_k$ such that $\upsilon_k(f) = 0$, which is equivalent by definition to $\upsilon_k'(f) = 0$. Now, again by definition:
\begin{equation*}
0 = \upsilon_k'(f) (n) = \sum_{v\in k\Z^d} f(n+v) = f(n)
\end{equation*}
for all $n \in I_k$. Thus $f = 0$ as desired. It is easy to check that $\upsilon_k(f) \in \ell^1(\Z^d_k)$ for $f\in \ell^1(\Z^d)$ --- in fact $\upsilon_k$ is easily seen to be of norm $1$.
\end{proof}

We will find it convenient to use the following notation across the remainder of this paper, with its first appearance in the proof of the theorem on pointwise strong operator continuity for the representations introduced in Notation (\ref{rep}), which follows.
\begin{notation}\label{cyclic-notation}
Let $d\in\N_\ast$, $k \in \Nbar^d$. For any $n \in \Z^d$, we define $\cyclic{n}{k}$ and $\cycle{n}{k}$ as the unique integers such that:
\begin{equation*}
n - \cycle{n}{k} = \cyclic{n}{k} \text{ with }\cycle{n}{k} \in k\Z^d \text{ and } \cyclic{n}{k} \in I_k \text{.}
\end{equation*}
This notation differs slightly from the usual meaning of the $\mod$ operator, as it refers to the partition $\mathscr{P}_k$ of Equation (\ref{partition-eq}) rather than the partition of $\Z^d$ in cosets for $k\Z^d$. However, this will prove a convenient notation, and is related to our choice in Notation (\ref{rep}). We relate our mod operation to the usual one in the following manner. For $d = 1$ an $k \in \N_\ast$, if $n \in \Z$ and $n = qk + r$ for $q,r \in \Z$ and $|r| < k$, then:
\begin{itemize}
\item If $r \in I_k$ then $\cyclic{n}{k} = r$,
\item If $r \in \left\{ \left\lfloor\frac{k-1}{2}\right\rfloor + 1,\ldots, k-1 \right\}$ then $\cyclic{n}{k} = r - k$,
\item If $r \in \left\{ -k+1,\ldots,\left\lfloor\frac{1-k}{2}\right\rfloor -1 \right\}$ then $\cyclic{n}{k} = r + k$.
\end{itemize}
If $d = 1$ then $\cyclic{n}{\infty} = n$ for all $n\in\Z$. The general computation of $\cyclic{n}{k}$, for arbitrary $d$, $k \in \Nbar^d$, and $n\in\Z^d$, is simply carried on component-wise using the above calculations for $d=1$.

Note that by construction, if $p_k$ is the inverse function of the restriction to $I_k$ of the canonical surjection $q_k : \Z^d\twoheadrightarrow \Z^d_k$, then $\cyclic{n}{k} = p_k\circ q_k(n)$ for all $n\in\Z^d$. In particular, we note for all $n,m \in \Z^d$, we have:
\begin{equation}\label{cyclic-computation-eq}
\cyclic{\cyclic{n}{k}-\cyclic{m}{k}}{k} = \cyclic{n-m}{k}\text{,}
\end{equation}
since $q_k\circ p_k$ is the identity on $\Z^d_k$.
\end{notation}

\begin{remark}\label{cocycle-cyclic-rmk}
Let $(k,\sigma) \in \Xi^d$ for some $d\in\N_\ast$. Identify $\sigma$ with its unique lift as a skew bicharacter of $\Z^d$. By construction, for all $n \in \Z^d_k$, we have:
\begin{multline*}
\sigma(m,n) = \sigma(\cyclic{m}{k},n) = \sigma(m,\cyclic{n}{k}) \\= \sigma(\cyclic{m}{k},\cyclic{n}{k})\text{.} 
\end{multline*}
\end{remark}

\subsection{Continuity}

We now can state the key result of this section:

\begin{theorem}\label{SOT-cont-thm}
Let $d\in\N_\ast$, $(k_\infty,\sigma_\infty) \in \Xi^d$. Let $(k_n,\sigma_n)_{n\in\N}$ be a sequence in $\Xi^d$ converging to $(k_\infty,\sigma_\infty)$. Let $f \in \ell^1(\Z^d)$. The sequence $\left(\pi_{k_n,\sigma_n}\left(\upsilon_{k_n}\left(f\right)\right)\right)_{n\in\N}$ converges to $\pi_{k_\infty,\sigma_\infty}(\upsilon_{k_\infty}(f) )$ in the strong operator topology of $\ell^2(\Z^d)$, where $\upsilon_c : \ell^1(\Z^d)\rightarrow\ell^1\left(\Z^d_c\right)$ is defined in Lemma (\ref{upsilon-lemma}) for all $c\in\Nbar^d$ while $\pi_{c,\sigma}$ is defined by Notation (\ref{rep}) for all $(c,\sigma)\in\Xi^d$.
\end{theorem}

\begin{proof}
For any $c \in \Nbar^d$, the canonical surjection $\Z^d \twoheadrightarrow \Z^d_c$ is denoted by $q_c$. We write $k_n = (k^1_n,\ldots,k^d_n)$ for all $n\in\N\cup\{\infty\}$. We identify the skew-bicharacters $\sigma_n$, for all $n\in\N\cup\{\infty\}$, with their unique lift to skew-bicharacters of $\Z^d$.

Let $m\in\Z^d$. Let $\delta_m \in \ell^1(\Z^d)$ be $1$ at $m$ and zero everywhere else, and let $(e_r)_{r\in\Z^d}$ be the canonical Hilbert basis of $\ell^2(\Z^d)$. A quick computation shows that $\upsilon_c\left(\delta_m\right) = \upsilon_c\left(\delta_{\cyclic{m}{c}}\right)$ for all $c \in \Nbar^d$. Moreover, from Equation (\ref{rep-eq}) in Notation (\ref{rep}) and using Equality (\ref{U-eq}) in Theorem (\ref{representation-thm}), as well as Equation (\ref{cyclic-computation-eq}) and Remark (\ref{cocycle-cyclic-rmk}), we have for all $r \in \Z^d$ and $n\in\Nbar$:
\begin{equation}\label{pi-action-eq}
\pi_{k_n,\sigma_n}(\upsilon_{k_n}(\delta_m)) e_r = \sigma_n(r,m)e_{\cyclic{r - m}{k_n} + \cycle{r}{k_n}}\text{.}
\end{equation}

We keep $m = (m_1,\ldots,m_d),r=(r_1,\ldots,r_d) \in \Z^d$ fixed for now. Let $j \in \{1,\ldots,d\}$. We have one of the following two cases:
\begin{itemize}
\item If $k_\infty^j = \infty$, then there exists $N_{r,m,j} \in \N$ such that for all $n\geq N_{r,m,j}$ we have $k_n^j > 2|r_j- m_j|$ and $k_n^j > |r_j|$, and thus:
\begin{align*}
\cyclic{r_j - m_j}{k^j_n} &= r_j - m_j  = \cyclic{r_j - m_j}{k_\infty^j}\\
\intertext{and}
\cycle{r_j}{k^j_n} &=\cycle{r_j}{k_\infty^j} =0
\end{align*}
for all $n \geq N_{r,m,j}$.

\item If $k_\infty^j \in \N_\ast$ then the sequence $(k^j_n)_{n\in\N}$, which lies in $\Nbar$, is eventually constant. Thus, there exists $N_{r,m,j}\in\N$ such that $k^j_n = k_\infty^j$ for all $n\geq N_{r,m,j}$. Consequently, we have $\cyclic{r_j - m_j}{k^j_n} = \cyclic{r_j - m_j}{k_\infty^j}$ and $\cycle{r_j}{k^j_n}=\cycle{r_j}{k_\infty^j}$ for all $j\geq N_{r,m,j}$.
\end{itemize}

Thus, there exists $N_{r,m} = \max\{N_{r,m,1},\ldots,N_{r,m,d}\}$ such that, for all $n\geq N_{r,m}$, we have $\cyclic{r - m}{k_n} = \cyclic{r - m}{k_\infty}$ and $\cycle{r}{k_n}=\cycle{r}{k_\infty}$. 

We thus have, by Equation (\ref{pi-action-eq}), for $n\geq N_{r,m}$:
\begin{multline*}
\pi_{k_n,\sigma_n}\left(\upsilon_{k_n}\left(\delta_m\right)\right)e_r - \pi_{k_\infty,\sigma_\infty}\left(\upsilon_{k_\infty}\left(\delta_m\right)\right)e_r\\ = (\sigma_n(r,m)-\sigma_\infty(r,m))e_{ \cyclic{r - m}{k_\infty} + \cycle{r}{k_\infty}}\text{.}
\end{multline*}
The sequence $(\sigma_n(r,m)-\sigma_\infty(r,m))_{n\in\N}$ converges to $0$ by assumption. Hence:
\begin{equation*}
\lim_{n\rightarrow\infty} \| \left(\pi_{k_n,\sigma_n}\left(\upsilon_{k_n}\left(\delta_m\right)\right) - \pi_{k_\infty,\sigma_\infty}\left(\upsilon_{k_\infty}\left(\delta_m\right)\right)\right)e_r\|_2 = 0\text{.}
\end{equation*}

Let now $\xi = (\xi_r)_{r\in\Z^d} \in\ell^2(\Z^d)$. Then, using Pythagorean theorem (since $\pi_{k_n,\sigma_n}-\pi_{k_\infty,\sigma_\infty}$ maps the canonical basis to an orthogonal family):
\begin{multline*}
\| \left(\pi_{k_n,\sigma_n}\left(\upsilon_{k_n}\left(\delta_m\right)\right) - \pi_{k_\infty,\sigma_\infty}\left(\upsilon_{k_\infty}\left(\delta_m\right)\right)\right)\xi \|_2^2 \\= \sum_{r\in\Z^d} |\xi_r|^2  \| (\sigma_n(r,m)-\sigma_\infty(r,m))e_{ \cyclic{r - m}{k_\infty} + \cycle{r}{k_\infty}}\|_2^2\\
= \sum_{r\in\Z^d} |\xi_r|^2|\sigma_n(r,m)-\sigma_\infty(r,m)|^2\text{.}
\end{multline*}
Now, since all bicharacters are valued in $\U$, we have:
\begin{equation*}
|\xi_r|^2|\sigma_n(r,m)-\sigma_\infty(r,m)|^2\leq 4|\xi_r|^2\text{ and } \sum_{r\in\Z^d} 4|\xi_r|^2= 4\|\xi\|_2^2 <\infty \text{.}
\end{equation*}
The Lebesgues dominated convergence theorem applies, and we conclude that:
\begin{equation}\label{wot-limit-eq}
\lim_{n\rightarrow\infty} \left\|\left(\pi_{k_n,\sigma_n}\left(\upsilon_{k_n}\left(\delta_m\right)\right)-\pi_{k_\infty,\sigma_\infty}\left(\upsilon_{k_\infty}\left(\delta_m\right)\right)\right)\xi \right\|_2 = 0\text{,}
\end{equation}
as desired.

Now, since for all $n\in\N\cup\{\infty\}$ the maps $\upsilon_{k_n}$ are continuous and linear, we immediately get:
\begin{multline*}
\left\|\left(\pi_{k_n,\sigma_n}\left(\upsilon_{k_n}\left(f\right)\right) - \pi_{k_\sigma,\sigma_\infty}\left(\upsilon_{k_\infty}\left(f\right)\right)\right)\xi \right\|_2 \\ \leq\sum_{m\in\Z^d} |f(m)| \left\|\left(\pi_{k_n,\sigma_n}\left(\upsilon_{k_n}\left(\delta_m\right)\right)-\pi_{k_\infty,\sigma_\infty}\left(\upsilon_{k_\infty}\left(\delta_m\right)\right)\right)\xi \right\|_2
\end{multline*}
for all $\xi \in \ell^2(\Z^d)$. Since $f \in \ell^1(\Z^d)$ and, for all $m\in\Z^d$:
\begin{equation*}
|f(m)|\left\|\left(\pi_{k_n,\sigma_n}\left(\upsilon_{k_n}\left(\delta_m\right)\right)-\pi_{k_\infty,\sigma_\infty}\left(\upsilon_{k_\infty}\left(\delta_m\right)\right)\right)\xi \right\|_2\leq2|f(m)|\|\xi\|_2\text{,}
\end{equation*}
the Lebesgues dominated convergence theorem implies that:
\begin{equation*}
\lim_{n\rightarrow\infty} \left\|\left(\pi_{k_n,\sigma_n}\left(\upsilon_{k_n}\left(f\right)\right) - \pi_{k_\infty,\sigma_\infty}\left(\upsilon_{k_\infty}\left(f\right)\right)\right)\xi \right\|_2 = 0
\end{equation*}
for all $\xi \in\ell^2(\Z^d)$ by Equation (\ref{wot-limit-eq}), as desired.
\end{proof}

It is easy to check that, in general, we can not strengthen Theorem (\ref{SOT-cont-thm}) by replacing the strong operator topology with the norm topology. This is, in turn, where using our notion of bridge, rather than the distance in norm between the Lipschitz balls, will prove very useful. As a hint of things to come, we prove the following results.

\begin{lemma}\label{trace-class-SOT-lemma}
Let $\Hilbert$ be a separable Hilbert space, $(T_n)_{n\in\N}$ be a sequence of operators on $\Hilbert$ and $T$ an operator on $\Hilbert$. Let $\B(\Hilbert)$ be the C*-algebra of all bounded linear operators on $\Hilbert$. The following assertions are equivalent:
\begin{enumerate}
\item The sequence $(T_n)_{n\in\N}$ converges to $T$ in the strong operator topology,
\item For all compact operators $A$ on $\Hilbert$, we have:
\begin{equation*}
\lim_{n\rightarrow\infty} \|(T_n-T)A\|_{\B(\Hilbert)} = 0 \text{,}
\end{equation*}
\item For all trace-class operators $A$ on $\Hilbert$, we have:
\begin{equation*}
\lim_{n\rightarrow\infty} \|(T_n-T)A\|_{\B(\Hilbert)} = 0 \text{.}
\end{equation*}
\end{enumerate}
\end{lemma}

\begin{proof}
Assume that $(T_n)_{n\in\N}$ converges to $T$ in the strong operator topology. By the uniform boundedness principle, $(T_n)_{n\in\N}$ is uniformly bounded. Let $M\in\R$ be chosen so that for all $n\in\N$, we have $\|T_n-T\|\leq M$. 

First, if $A$ is a rank one operator and if $\xi \in \Hilbert$ is a unit vector which spans the range of $A$, then for all unit vector $\eta\in \Hilbert$ there exists $\lambda\in\C$ such that $A\eta=\lambda \xi$. Then, since $\|A\eta\|_{\Hilbert} = \|\lambda \xi\|_{\Hilbert} \leq \|A\|_{\B(\Hilbert)}$, we have $|\lambda|\leq \|A\|_{\B(\Hilbert)}$, and thus:
\begin{equation}\label{sot-eq1}
\|(T_n-T)A\|_{\B(\Hilbert)} \leq \|A\|_{\B(\Hilbert)} \|(T_n-T)\xi\|_\Hilbert \stackrel{n\rightarrow\infty}{\longrightarrow} 0 \text{.}
\end{equation}

Of course, Equation (\ref{sot-eq1}) holds if $A = 0$ as well.

Let $A$ now be an arbitrary compact operator, and let $\varepsilon > 0$. Since $\Hilbert$ is separable, there exists a finite rank operator $A_\varepsilon$ on $\ell^2(\Z^d)$ such that:
\begin{equation*}
\|A-A_\varepsilon\|_{\B(\Hilbert)} < \frac{1}{2M}\varepsilon\text{.}
\end{equation*}
Since $A_\varepsilon$ is finite rank, by Equation (\ref{sot-eq1}) (and a trivial induction), there exists $N \in \N$ such that for all $n\geq N$ we have:
\begin{equation*}
\|(T_n-T)A_\varepsilon\|_{\B(\Hilbert)} \leq \frac{1}{2M}\varepsilon \text{.}
\end{equation*}
Thus, for $n\geq N$:
\begin{equation*}
\begin{split}
\|(T_n-T)A\|_{\B(\Hilbert)}&\leq \|T_n-T\|_{\B(\Hilbert)}\|A-A_\varepsilon\|_{\B(\Hilbert)} + \|(T_n-T)A_\varepsilon\|_{\B(\Hilbert)}\\
&\leq \frac{1}{2}\varepsilon + \frac{1}{2}\varepsilon = \varepsilon \text{.}
\end{split}
\end{equation*}
This proves our first implication. Our second implication is trivial, since trace-class operators are compact. Last, suppose that (3) holds. For any unit vector $\xi \in \Hilbert$,  we let $A$ be the be projection on the span of $\xi$. Thus, $A$ is a trace-class operator, and then for all $n\in\N$:
\begin{equation*}
\|(T-T_n)A\|_{\B(\Hilbert)}=\|(T-T_n)\xi\|_\Hilbert\text{,}
\end{equation*}
so by (3), we have $\lim_{n\rightarrow\infty} \|(T-T_n)\xi\|_\Hilbert = 0$. As $\xi$ is an arbitrary unit vector in $\Hilbert$, this is enough to conclude that $(T_n)_{n\in\N}$ converges in the strong operator topology to $T$, as desired.
\end{proof}

\begin{corollary}\label{compact-sot-corollary}
Let $d\in\N_\ast$, $(k,\sigma) \in \Xi^d$ and $T$ be a compact operator on $\ell^2(\Z^d)$. Let $(k_n,\sigma_n)_{n\in\N}$ be a sequence in $\Xi^d$ converging to $(k,\sigma)$. If $f \in \ell^1(\Z^d)$, then the sequence $\left(\pi_{k_n,\sigma_n}\left(\upsilon_{k_n}(f)\right)T \right)_{n\in\N}$ converges in norm to $\pi_{k,\sigma}\left(\upsilon_k(f)\right)T$.
\end{corollary}

\begin{proof}
Apply Lemma (\ref{trace-class-SOT-lemma}) to Theorem (\ref{SOT-cont-thm}).
\end{proof}

Theorem (\ref{SOT-cont-thm}) can be used to show that for $f \in \ell^1(\Z^d)$, the map:
\begin{equation*}
(k,\sigma)\in\Xi^d\mapsto \|\pi_{k,\sigma}(\upsilon_k(f))\|_{\B^d}
\end{equation*}
is lower semi-cont\-inuous. We proved, in fact, that this map is continuous in \cite[Theorem 2.6]{Latremoliere05}, which is an essential building block for our current work as well.

\begin{theorem}[\cite{Latremoliere05}]\label{norm-cont-thm}
Let $(k_n,\sigma_n)_{n\in\N}$ be a sequence in $\Xi^d$ converging to some $(k,\sigma)\in \Xi^d$. If $f \in \ell^1(\Z^d)$ then $(\|\pi_{k_n,\sigma_n}(\upsilon_{k_n}(f))\|_{\B^d})_{n\in\N}$ converges to $\|\pi_{k,\sigma}(\upsilon_{k}(f))\|_{\B^d}$, where  $\pi_{c,\theta}$ are the representations introduced in Notation (\ref{rep}) for all $(c,\theta)\in\Xi^d$.
\end{theorem}

Though Theorem (\ref{SOT-cont-thm}) only asserts pointwise continuity for the norms, we can actually improve the situation using the following lemma, which is the first part of \cite[Lemma 10.1]{Rieffel00}. We include its proof for the reader's convenience.

\begin{lemma}\label{cont-lemma}
Let $\mathsf{m}$ be a norm on a subspace $V$ of $\ell^1(\Z^d)$. For all $(k,\sigma)\in\Xi^d$, we assume given a norm $\mathsf{n}_{k,\sigma}$ of $\upsilon_k(V)$ such that $\mathsf{n}_{k,\sigma}\circ\upsilon_k\leq\mathsf{m}$ and such that, for all $v \in V$, the map $(k,\sigma)\in\Xi^d \mapsto \mathsf{n}_{k,\sigma}\left(\upsilon_k(v)\right)$ is continuous. Then:
\begin{equation}\label{cont-lem-eq}
(k,\sigma,v) \in \Xi^d\times V \longmapsto \mathsf{n}_{k,\sigma}(\upsilon_k(v))
\end{equation}
is continuous, when $V$ is endowed with the norm $\mathsf{m}$ and $\Xi^d\times V$ with the product topology.
\end{lemma}

\begin{proof}
For any $(k,\sigma),(c,\theta)\in\Xi^d$ and $a,b \in V$, we have:
\begin{equation*}
\begin{split}
|\mathsf{n}_{c,\theta}(\upsilon_c(b)) - \mathsf{n}_{k,\sigma}(\upsilon_k(a))|&\leq |\mathsf{n}_{c,\theta}(\upsilon_c(b)) - \mathsf{n}_{c,\theta}(\upsilon_c(a))| + | \mathsf{n}_{c,\theta}(\upsilon_c(a)) - \mathsf{n}_{k,\sigma}(\upsilon_k(a)) | \\
&\leq \mathsf{n}_{c,\theta}(\upsilon_c(b-a)) + | \mathsf{n}_{c,\theta}(\upsilon_c(a)) - \mathsf{n}_{k,\sigma}(\upsilon_k(a)) |\\
&\leq \mathsf{m}(b-a) + |\mathsf{n}_{c,\theta}(\upsilon_c(a)) - \mathsf{n}_{k,\sigma}(\upsilon_k(a)) |\text{.}
\end{split}
\end{equation*}
Let $\varepsilon > 0$. By assumption, there exists an open neighborhood $U$ of $(k,\sigma)$ in $\Xi^d$ such that, for all $(c,\theta) \in U$, we have:
\begin{equation*}
|\mathsf{n}_{c,\theta}(\upsilon_c(a)) - \mathsf{n}_{k,\sigma}(\upsilon_k(a)) | < \frac{1}{2}\varepsilon\text{.}
\end{equation*}
Thus, if $U_2$ is the open ball of center $a$ and radius $\frac{1}{2}\varepsilon$ for the norm $\mathsf{m}$, we then have that for all $(c,\theta,b) \in U\times U_2$:
\begin{equation*}
|\mathsf{n}_{c,\theta}(\upsilon_c(b)) - \mathsf{n}_{k,\sigma}(\upsilon_k(a))| < \varepsilon \text{,}
\end{equation*}
as desired.
\end{proof}

\begin{corollary}\label{norm-cont-corollary}
Let $d\in\N_\ast$. The map:
\begin{equation*}
(k,\sigma,f) \in \Xi^d\times \ell^1(\Z^d) \longmapsto \|\pi_{c,\theta}(\upsilon_k(f))\|_{\B^d}
\end{equation*}
is continuous when $\ell^1(\Z^d)$ is endowed with $\|\cdot\|_1$ and $\Xi^d\times\ell^1(\Z^d)$ with the product topology.
\end{corollary}

\begin{proof}
Apply Lemma (\ref{cont-lemma}) to the field of norms given by Theorem (\ref{SOT-cont-thm}), where $V$ is indeed finite dimensional, and where $\mathsf{m} = \|\cdot\|_1$.
\end{proof}

In practice, Corollary (\ref{norm-cont-corollary}) will be used when we restrict our attention to a finite dimensional subspace of $\ell^1(\Z^d)$. Indeed, using the continuity of the norms, as well as the definition of Lip-norms, we proved in \cite[Proposition 3.10]{Latremoliere05} a result concerning the continuity of Lip-norms over finite dimensional subspaces of $\ell^1(\Z^d)$, which we shall need for our work as well. Putting it together with Lemma (\ref{cont-lemma}), we have:

\begin{theorem}\label{lip-norm-cont-thm}
Let $d\in\N_\ast$ and $S$ be a finite set in $\Z^d$ which does not contain $0$. Define:
\begin{equation*}
V = \{ f \in \ell^1(\Z^d) : \forall n \not\in S \quad f(n) = 0 \} \text{.}
\end{equation*}
Then $V$ is a finite dimensional subspace of $\ell^1(\Z^d)$. Let $l$ be a continuous length function on $\U^d$. The map:
\begin{equation*}
(k,\sigma,f) \in \Xi^d\times V \longmapsto \Lip_{l,k,\sigma}(\upsilon_{k}(f))
\end{equation*}
is continuous.
\end{theorem}

\begin{proof}
By \cite[Proposition 3.10]{Latremoliere05}, for a fixed $f \in V$, the function
\begin{equation*}
(k,\sigma)\in\Xi^d \mapsto \Lip_{k,\sigma}(\upsilon_k(f))
\end{equation*}
is continuous. Now, for any $f \in V$ and $\omega\in\U^d$, we set:
\begin{equation*}
\alpha^\omega(f) : n \in \Z^d \mapsto \omega^n f(n)\text{.}
\end{equation*}
We see that $\alpha$ thus defined is an ergodic action of $\U^d$ on $\ell^1(\Z^d)$. For any $f\in V$, we set:
\begin{equation*}
\mathsf{m}(f) = \sup\left\{ \frac{\|f-\alpha^\omega(f)\|_1}{l(\omega)} : \omega \in \U^d \right\} \text{.}
\end{equation*} 
By construction, for all $(k,\sigma,f)\in \Xi^d\times V$ we have $\Lip_{l,k,\sigma}(f) \leq \mathsf{m}(f)$. Indeed, we note that by Theorem-Definition (\ref{dual-action-thmdef}), the dual action $\alpha_{k,\sigma}$ restricted to $\U^d_k$ and the action $\alpha$ agree when applied to elements of $\ell^1(\Z^d)$ --- and thus in particular to elements in $V$.

Moreover, $\mathsf{m}$ takes finite values on $V$, which is seen using the finite dimensionality of $V$ as follows. One checks that $\mathsf{m}(\delta_p) < \infty$ for all $p \in S$ with $\delta_p$, as usual, the Dirac measure at $p$. Thus $\mathsf{m}$ takes finite values on $V$ (as $\mathsf{m}$ is a seminorm and $V$ is the finite dimensional space spanned by $\{\delta_p : p \in S\}$). It is then easy to check that $\mathrm{m}$ is a norm as $\alpha$ only fixes the elements of the space $\C\delta_0$ and since $0\not\in S$. Thus Lemma (\ref{cont-lemma}) applies to give our theorem.
\end{proof}

Finite dimensionality and the exclusion of $0$ from the support of the elements of $V$ is used in Theorem (\ref{lip-norm-cont-thm}) so that we can find a norm which dominates all the Lip-norms; in general $\mathsf{m}(f)$, as defined in the proof of Theorem (\ref{lip-norm-cont-thm}), for $f\in\ell^1(\Z^d)$, is neither finite nor zero only at zero, and thus does not define a norm.

\section{Quantum Propinquity and Quantum Tori}

We establish the main convergence result of this paper in this section. We only need bridges with $\B^d = \B(\ell^2(\Z^d))$ as our bridge C*-algebra, the various representations $\pi_{k,\sigma}$ ($(k,\sigma)\in\Xi^d$) of Notation (\ref{rep}) as our unital *-mono-morphisms, and pivots elements which are trace class and diagonal in the canonical basis of $\ell^2(\Z^d)$. We begin with introducing our pivot elements and obtain an estimate on the norm of commutators between our pivot elements and our C*-algebras.

\subsection{Pivot Elements}

 To choose these pivot elements, we will use the following notations, and derive an estimate on the norm of some commutators.

\begin{notation}
Let $d\in\N_\ast$. Let $(\lambda_n)_{n\in\Z^d}$ be a family of complex numbers indexed by $\Z^d$. The operator $\Diag{\lambda_n}{n\in\Z^d}$ on $\ell^2(\Z^d)$ is defined by setting for all $n \in \Z^d$:
\begin{equation*}
\Diag{\lambda_n}{n\in\Z^d}e_n = \lambda_n e_n\text{,}
\end{equation*}
where $(e_n)_{n\in\Z^d}$ is the canonical Hilbert basis of $\ell^2(\Z^d)$ defined in Notation (\ref{canonical-basis-notation}).
\end{notation}

\begin{notation}
For any $d\in\N_\ast$ and any $n = (n_1,\ldots,n_j)\in\Z^d$, we define:
\begin{equation*}
|n| = \sum_{j=1}^d |n_j|\text{.}
\end{equation*}
We note that $|\cdot|$ thus defined is the length  function on $\Z^d$ associated with the canonical generators of $\Z^d$. Thus, in particular, for any $n,m\in\Z^d$ we have $| |n| - |m| | \leq |n-m| \leq |n| + |m|$. 
\end{notation}

\begin{notation}\label{wedge-notation}
For any $d\in\N_\ast$, and for any $k = (k_1,\ldots,k_d)\in \Nbar^d$, we denote by $\wedge k$ the element of $\Nbar$ defined as:
\begin{equation*}
\wedge k = \min \{ |n| : n \not\in I_k \} = \min\left\{ \left\lceil\frac{k_j-1}{2}\right\rceil : j=1,\ldots,d \right\} + 1\text{.}
\end{equation*}
\end{notation}

\begin{notation}\label{weight-notation}
Let $M,N\in\N_\ast$ be given. We define, for all $d\in\N_\ast$:
\begin{equation*}
w_{N,M} : n \in \Z^d \longmapsto \begin{cases}
1 &\text{if $|n| \leq N$,}\\
\frac{M + N - |n|}{M}&\text{ if $N\leq |n| \leq M + N$,}\\
0 &\text{otherwise.}
\end{cases}
\end{equation*}
\end{notation}

The proper choice of a bridge for a given pair of quantum tori, or a pair of a fuzzy torus with a quantum torus, will depend on the following estimate.

\begin{theorem}\label{qt-commutator-thm}
Let $d\in\N_\ast$, $k\in\Nbar^d$, and $\sigma$ a skew-bicharacter of $\Z^d_k$. Let $N,M \in \N_\ast$ such that $N+M < \wedge k$. Define, using Notation (\ref{weight-notation}):
\begin{equation*}
\omega_{N,M} = \Diag{ w_{N,M}(n) }{n\in\Z^d}\text{.}
\end{equation*}
Then $\omega_{N,M}$ is a finite rank operator such that, for all $m \in I_k\subseteq \Z^d$, we have:
\begin{equation*}
\left\| \left[\omega_{N,M},\pi_{k,\sigma}\left(\delta_{q_k(m)}\right)\right]\right\|_{\B^d} \leq \frac{|m|}{M} \text{,}
\end{equation*}
where $q_k : \Z^d\rightarrow \Z^d_k$ is the canonical surjection and $\pi_{c,\theta}$ is given by Notation (\ref{rep}) for all $(c,\theta)\in\Xi^d$.
\end{theorem}

\begin{proof}
Let $m \in I_k$. Let $\upsilon_k : \ell^1(\Z^d)\rightarrow\ell^2(\Z^d_k)$ be given by Lemma (\ref{upsilon-lemma}). An easy computation shows that $\upsilon_k(\delta_m) = \delta_{q_k(m)}$, and moreover, for all $n\in\Z^d$:
\begin{equation}\label{comm-pi-computation-eq}
\pi_{k,\sigma}(\upsilon_{k}(\delta_m)) e_r = \sigma(r,m)e_{\cyclic{r - m}{k} + \cycle{r}{k}}
\end{equation}
by Equation (\ref{pi-action-eq}) in the proof of Theorem (\ref{SOT-cont-thm}).

We thus easily compute that, for all $n\in\Z^d$:
\begin{multline*}
\left[\omega_{N,M},\pi_{k,\sigma}\left(\delta_{q_k(m)}\right)\right]e_n = \sigma(n,m)\\\left(w_{N,M}(\cyclic{n - m}{k} + \cycle{n}{k})- w_{N,M}(n)\right) e_{\cyclic{n - m}{k}+\cycle{n}{k}}\text{.}
\end{multline*}

Using Notation (\ref{weight-notation}), we first observe that if $n \not\in I_k$, then, $\cyclic{n-m}{k}+\cycle{n}{k} \not\in I_k$ either, and thus we have:
\begin{equation}\label{comm-weight-eq0}
w_{N,M} (n) = w_{N,M}(\cyclic{n-m}{k}+\cycle{n}{k}) = 0\text{,}
\end{equation}
since $|n|$ and $|\cycle{n}{k}+\cyclic{n}{k}|$ are both greater than $\wedge k$, and thus greater than $N+M$.

We now work now with $n \in I_k$ (so $\cycle{n}{k} = 0$). Let $j \in \{1,\ldots,d\}$ and write $n=(n_1,\ldots,n_d)$, $m=(m_1,\ldots,m_d)$,$k=(k_1,\ldots,k_d)$.  We have the following cases:
\begin{itemize}
\item If $n_j - m_j \in I_{k_j}$ then $\cyclic{n_j - m_j}{k_j} = n_1 - m_j$, so:
\begin{equation*}
\left| |\cyclic{n_j - m_j}{k_1}| - |n_j| \right| = | |n_j - m_j|-|n_j| |\leq |m_j|\text{.}
\end{equation*}
\item If $n_j - m_j < \frac{-k_j}{2}$ then:
\begin{equation*}
\left|\cyclic{n_j-m_j}{k_j}\right| = n_j - m_j + k_j\text{.}
\end{equation*}
Then $2n_j - 2m_j < -k_j$, so $k_j - m_j + 2n_j \leq m_j$. On the other hand, note that since $n_j, m_j \in I_{k_j}$ and $n_j-m_j<\frac{-k_j}{2}$, we must have $n_j \leq 0$ and $m_j \geq 0$; thus:
\begin{equation*}
\left|\left|\cyclic{n_j-m_j}{k_j}\right| - |n_j|\right| = n_j-m_j+k_j+n_j = 2n_j-m_j+k_j \leq m_j = |m_j|\text{.}
\end{equation*}
\item Similarly, $n_j - m_j \geq \frac{k_j}{2}$, then we get:
\begin{equation*}
\left|\cyclic{n_j-m_j}{k_j} \right| = k_j-n_j+m_j\text{,}
\end{equation*}
while $2n_j-2m_j \geq k_j$, i.e. $2n_j-m_j-k_j \geq -m_j$. In this case, we must have $n_j\geq 0$ and $m_j\leq 0$; thus:
\begin{equation*}
\left|\left|\cyclic{n_j-m_j}{k_j}\right| - |n_j|\right| = k_j-n_j+m_j-n_j = -2n_j+m_j+k_j \leq -m_j = |m_j|\text{.}
\end{equation*}
\end{itemize}

Thus, we always have, for all $n,m\in I_k$:
\begin{equation}\label{comm-weight-eq1}
\left| \left|\cyclic{n-m}{k}\right| -|n| \right| \leq |m| \text{.}
\end{equation} 

Consequently, by Notation (\ref{weight-notation}) and Equations (\ref{comm-weight-eq0}), (\ref{comm-weight-eq1}), we have the following cases:

\begin{itemize}
\item If $\{|n|,|\cyclic{n-m}{k}|\} \subseteq \{0,\ldots,N\}$ or $\{|n|,|\cyclic{n-m}{k}\}\subseteq \{N+M,\ldots\}$ then:
\begin{equation*}
|w_{N,M}(\cyclic{n-m}{k}) - w_{N,M}(n)| = 0\leq \frac{|m|}{M}\text{.}
\end{equation*} 

\item If $\{|n|,|\cyclic{n-m}{k}|\} \subseteq \{N,\ldots,M+N\}$ then:
\begin{equation*}
\begin{split}
|w_{N,M}(\cyclic{n-m}{k}) - w_{N,M}(n)| &= \left|\frac{ |n| - |\cyclic{n-m}{k}|}{M}\right|\\
&\leq \frac{|m|}{M}\text{.}
\end{split}
\end{equation*}
\item If $0\leq |n| \leq N$ and $N\leq\cyclic{n-m}{k}\leq N+M$ then:
\begin{equation*}
\begin{split}
|w_{N,M}(\cyclic{n-m}{k}) - w_{N,M}(n)| &= 1 - \frac{N+M-|\cyclic{n-m}{k}|}{M}\\
&\leq \frac{|\cyclic{n-m}{k}| - N}{M}\\
&\leq \frac{|\cyclic{n-m}{k}|-|n|}{M} \leq \frac{|m|}{M}\text{.}
\end{split}
\end{equation*}
\item The same computation as above shows that if $|\cyclic{n-m}{k}|\leq N$ and $N\leq |n|\leq N+M$, then again:
\begin{equation*}
|w_{N,M}(\cyclic{n-m}{k}) - w_{N,M}(n)| \leq \frac{|m|}{M}\text{.}
\end{equation*} 
\item Assume now that $|n|\geq N+M$ while $N\leq|\cyclic{n-m}{k}|\leq N+M$. Then:
\begin{equation*}
\begin{split}
|w_{N,M}(\cyclic{n-m}{k}) - w_{N,M}(n)| &= \frac{M+N-|\cyclic{n-m}{k}|}{M}\\
&\leq \frac{|n|-|\cyclic{n-m}{k}|}{M} \leq \frac{|m|}{M}\text{.}
\end{split}
\end{equation*}
\item A similar computation shows that if $N\leq |n|\leq M + N$ and $|\cyclic{n-m}{k}|\geq M+N$ then:
\begin{equation*}
|w_{N,M}(\cyclic{n-m}{k}) - w_{N,M}(n)| \leq \frac{|m|}{M}\text{.}
\end{equation*} 
\item Last, if either $|n|\leq N$ and $|\cyclic{n-m}{k}|\geq N+M$, or $|n|\geq N+M$ and $|\cyclic{n-m}{k}\leq N$ then:
\begin{equation*}
\begin{split}
\frac{|m|}{M}&\geq \frac{| |n| - |\cyclic{n-m}{k}| |}{M} \\
&\geq 1 = |w_{N,M}(\cyclic{n-m}{k}) - w_{N,M}(n)|\text{.}
\end{split}
\end{equation*}
\end{itemize}

Thus, by Equation (\ref{comm-pi-computation-eq}), for all $n,m\in I_k$ we have:
\begin{equation*}
\left\|\left[\omega_{N,M},\pi_{k,\sigma}\left(\delta_{q_k(m)}\right)\right]e_n\right\|_2 \leq \frac{|m|}{M}\text{,}
\end{equation*}
while for all $m\in I_k, n\not\in I_k$ we have:
\begin{equation*}
\left\|\left[\omega_{N,M},\pi_{k,\sigma}\left(\delta_{q_k(m)}\right))\right]e_n\right\|_2 = 0\text{.}
\end{equation*}
Let $\xi = (\xi_n)_{n\in\Z^d} \in \ell^2(\Z^d)$. Then for all $m\in I_k$ we have, using the Pythagorean Theorem (as $\left[\omega_{N,M},\pi_{k,\sigma}\left(\delta_{q_k(m)}\right)\right]$ maps the canonical basis to an orthogonal family by Equation (\ref{pi-action-eq})):
\begin{equation*}
\begin{split}
\left\|\left[\omega_{N,M},\pi_{k,\sigma}\left(\delta_{q_k(m)}\right)\right]\xi\right\|_2^2&= \sum_{n\in \Z^d} |\xi_n|^2\left\|\left[\omega_{N,M},\pi_{k,\sigma}\left(\delta_{q_k(m)}\right)\right]e_n\right\|_2^2\\
&=\sum_{n\in I_k} |\xi_n|^2\left\|\left[\omega_{N,M},\pi_{k,\sigma}\left(\delta_{q_k(m)}\right)\right]e_n\right\|_2^2\\
&\leq \sum_{n\in I_k} |\xi_n|^2\left(\frac{|m|^2}{M^2} \right) \leq \left(\frac{|m|^2}{M^2}\right)\|\xi\|_2^2 \text{.}
\end{split}
\end{equation*}
Hence our result is proven.
\end{proof}

%%%%%%%%%%%%%%%%%%%%%

\subsection{A new proof of Convergence}
We are now almost ready for the main result of this paper. As its proof is rather long, we isolate two lemmas which will prove useful in our construction.

\begin{notation}
Let $\mathfrak{L}_1$ be the ideal of trace-class operators on $\ell^2(\Z^d)$, whose norm is denoted by $\|\cdot\|_{\mathfrak{L}_1}$. We denote by $\tr$ the standard trace on the C*-algebra $\B(\ell^2(\Z^d))$ of bounded operators on $\ell^2(\Z^d)$ and we recall that $\|A\|_{\mathfrak{L}_1} = \tr(|A|)$ for all $A\in\mathfrak{L}_1$, with $|A|=\left(A^\ast A\right)^{\frac{1}{2}}$.
\end{notation}

\begin{lemma}\label{density-trace-class-lemma}
For all $N\in\N$, let $P_N$ be the projection on the span of $\{ e_n : |n| \leq N\}$ in $\ell^2(\Z^d)$, where $(e_m)_{m\in\Z^d}$ is the canonical Hilbert basis of $\ell^2(\Z^d)$. For any positive trace class operator $A$ on $\ell^2(\Z^d)$ with trace $1$, and for any $\varepsilon > 0$, there exists $N \in \N$  and a positive trace class operator $B$ of trace $1$ such that $P_N B P_N = B$ and $\|A-B\|_{\mathfrak{L}_1} < \varepsilon$. Note that for such an $N$, $\varepsilon$ and $B$, we have $P_n B P_n = B$ for all $n\geq N$, since $P_NP_n = P_N$.
\end{lemma}

\begin{proof}
For any pair of vectors $\zeta,\chi \in \ell^2(\Z^d)$, we denote the inner product of $\ell^2(\Z^d)$ of $\zeta$ and $\chi$ as $\inner{\zeta}{\chi}$.

Since $A$ is a positive compact operator with trace $1$, let $(\xi_n)_{n\in\N}$ be an orthonormal basis of $\ell^2(\Z^d)$ such that:
\begin{equation*}
A = \sum_{n\in\Z^d} \lambda_n \inner{\cdot}{\xi_n}\xi_n\text{,}
\end{equation*}
where $\lambda_n \in [0,1]$ for all $n\in\N$, and $\sum_{n\in\Z^d}\lambda_n = 1$. 

Let $N_1 \in \N$ be chosen so that $\sum_{\{n \in \Z^d : |n|\geq m\}} \lambda_n < \frac{1}{4}\varepsilon$ for all $m\geq N_1$. Let:
\begin{equation*}
A' = \sum_{\{n\in\Z^d : |n|\leq N_1\}} \lambda_n \inner{\cdot}{\xi_n}\xi_n\text{.}
\end{equation*}
By construction:
\begin{equation*}
\|A-A'\|_{\mathfrak{L}_1} <\frac{\varepsilon}{4}\text{.}
\end{equation*}

Let $N\in \N$ so that, for all $m\in\Z^d$ with $|m|\leq N_1$, we have:
\begin{equation}\label{trace-eq0}
\left\| P_N\xi_m - \xi_m \right\|_2\leq \frac{\varepsilon}{8}\text{,}
\end{equation}
which can be found since $\{ \xi_m : m\in\Z^d\text{ and }|m|\leq N_1\}$ is finite and $(P_n)_{n\in\N}$ converges to the identity in the strong operator topology.

Let:
\begin{equation*}
\eta_m = P_N\xi_m = \sum_{\{n\in \Z^d : |n| \leq N\}} \inner{\xi_m}{e_n}e_n
\end{equation*}
for all $m\in\Z^d$, $|m|\leq N_1$.

Define:
\begin{equation*}
A'' = \sum_{\{m\in\Z^d : |m| \leq N_1\}} \lambda_m \inner{\cdot}{\eta_m}\eta_m\text{,}
\end{equation*}
and note that $A''$, as the sum of positive operators, is a positive operator; moreover $A''$ is a finite rank operator such that $P_{N}A''P_{N}=A''$.

We now have:
\begin{equation}\label{trace-eq1}
\begin{split}
\|A-A''\|_{\mathfrak{L}_1} &\leq \|A - A'\|_{\mathfrak{L}_1} + \|A'-A''\|_{\mathfrak{L}_1}\\
&\leq \frac{\varepsilon}{4} + \left\|\sum_{\{n \in \Z^d:|n|\leq N_1\}} \lambda_n \left(\inner{\cdot}{\xi_n}\xi_n - \inner{\cdot}{\eta_n}\eta_n\right)\right\|_{\mathfrak{L}_1}\\
&\leq \frac{\varepsilon}{4} + \sum_{\{n\in\Z^d : |n|\leq N_1\}} \lambda_n \left( \|\inner{\cdot}{\xi_n-\eta_n}\xi_n\|_{\mathfrak{L}_1} + \| \inner{\cdot}{\eta_n}(\xi_n-\eta_n) \|_{\mathfrak{L}_1}   \right)\text{.}
\end{split}
\end{equation}

Now, for any $\chi,\zeta\in\ell^2(\Z^d)$, we have $\left(\inner{\cdot}{\chi}\zeta\right)^\ast = \inner{\cdot}{\zeta}\chi$, so that:
\begin{equation*}
\left|\inner{\cdot}{\chi}\zeta\right| = \|\chi\|_2\|\zeta\|_2\inner{\cdot}{\frac{\zeta}{\|\zeta\|_2}}\frac{\zeta}{\|\zeta\|_2}\text{,}
\end{equation*}
so:
\begin{equation}\label{trace-eq2}
\|\inner{\cdot}{\chi}\zeta\|_{\mathfrak{L}_1} = \|\chi\|_2\|\zeta\|_2\text{.}
\end{equation}
From Inequalities (\ref{trace-eq0}),(\ref{trace-eq1}), and (\ref{trace-eq2}), we conclude:
\begin{equation*}
\|A-A''\|_{\mathfrak{L}_1} \leq \frac{1}{4}\varepsilon + \sum_{\{n\in\Z^d:|n|\leq N_1\}} \lambda_n\left(\frac{\varepsilon}{8} + \frac{\varepsilon}{8}\right) \leq \frac{1}{2}\varepsilon\text{.}
\end{equation*}

To conclude, let $B = \frac{1}{\|A''\|_{\mathfrak{L}_1}}A''$. Now, $B$ satisfies $P_{N}BP_{N} = B$ as $A''$ does, and since $A''$ is positive, so is $B$. Moreover, $\tr(B) = \|B\|_{\mathfrak{L}_1} = 1$ as, again, $B$ is positive.

Moreover:

\begin{equation*}
\begin{split}
\|A-B\|_{\mathfrak{L}_1} &\leq \|A-A''\|_{\mathfrak{L}_1} + \left\|A'' - \frac{1}{\|A\|''}A''\right\|_{\mathfrak{L}_1}\\
&\leq \|A-A''\|_{\mathfrak{L}_1} + |1-\|A''\|_{\mathfrak{L}_1}|\\
&= \|A-A''\|_{\mathfrak{L}_1} + |\|A\|_{\mathfrak{L}_1} - \|A''\|_{\mathfrak{L}_1}|\\
&\leq 2\|A-A''\|_{\mathfrak{L}_1} \leq \varepsilon\text{.}
\end{split}
\end{equation*}
This concludes our lemma.
\end{proof}

We can now apply Lemma (\ref{density-trace-class-lemma}) for the following construction, which is essential to the computation of the height of the bridges to come.

\begin{lemma}\label{F_N-lemma}
Let $\mathfrak{L}_1^{1+}$ be the set of all positive trace class operators on $\ell^2(\Z^d)$ of trace $1$. For any $A\in\mathfrak{L}_1^{1+}$, we define:
\begin{equation*}
\psi_A : T \in \B(\ell^2(\Z^d)) \longmapsto \tr(AT)\text{.}
\end{equation*}
Let $\sigma$ be a skew-bicharacter of $\Z^d$ and $l$ be a continuous length function on $\U^d$. Let $\varepsilon>0$. There exists $N\in\N$ and a finite set $\mathfrak{F}_N$ of $\mathfrak{L}_1^+$ such that:
\begin{equation*}
\Haus{\Kantorovich{\Lip_{l,\infty^d,\sigma}}}(\StateSpace(\A_{\infty^d,\sigma}),\{ \psi_A\circ\pi_{\infty^d,\sigma} : A\in \mathfrak{F}_N \}) \leq \varepsilon
\end{equation*}
and
\begin{equation*}
\forall A \in \mathfrak{F}_N\quad P_N A P_N = P_N A = A P_N = A
\end{equation*}
where $P_N$ is the projection of $\ell^2\left(\Z^d\right)$ on the span of $\{e_n : |n|\leq N\}$, with $(e_n)_{n\in\Z^d}$ the canonical Hilbert basis of $\ell^2\left(\Z^d\right)$.
\end{lemma}

\begin{proof}
By \cite[Proposition VII 5.2]{Fell88}, since $\pi_{\infty^d,\sigma}$ is faithful, the set:
\begin{equation*}
\{ \psi_A\circ\pi_{\infty^d,\sigma} : A \in \mathfrak{L}_1^{1+}\}
\end{equation*}
is a weak* dense subset of $\StateSpace(\A_{\infty^d,\sigma})$. Since $\Kantorovich{\Lip_{l,\infty^d,\sigma}}$ metrizes the weak* topology of $\StateSpace(\A_{\infty^d,\sigma})$, and this space is compact, there exists a \emph{finite} subset $\mathfrak{F} \subseteq \mathfrak{L}_1^{1+}$ such that, for all $\varphi \in \StateSpace(\A)$, there exists $A\in\mathfrak{F}$ such that:
\begin{equation}\label{qt-states-eq1}
\Kantorovich{\Lip_{l,\infty^d,\sigma}}(\varphi,\psi_A\circ\pi_{\infty^d,\sigma}) < \frac{\varepsilon}{2} \text{.}
\end{equation}

Let $D = \diam{\StateSpace(\A_{\infty^d,\sigma})}{\Kantorovich{\Lip_{l,\infty^d,\sigma}}}$ --- since $(\StateSpace(\A_{\infty^d,\sigma}),\Kantorovich{\Lip_{l,\infty^d,\sigma}})$ is compact, $D<\infty$. For any $N \in \N$, let $P_N$ the projection on the span of $\{e_j : |j|\leq N\}$ in $\ell^2(\Z^d)$. By Lemma (\ref{density-trace-class-lemma}), since $\mathfrak{F}$ is finite, there exists $N\in\N$ and a finite set $\mathfrak{F}_N$ such that:
\begin{itemize}
\item for all $A\in\mathfrak{F}$, there exists $B \in \mathfrak{F}_N$ such that $\|A-B\|_{\mathfrak{L}_1}<\frac{\varepsilon}{2 D}$,
\item  For all $B\in\mathfrak{F}_N$, we have $P_N B P_N = B P_N = P_N B = B$.
\end{itemize}

Let $A \in \mathfrak{F}$, and choose $B\in \mathfrak{F}_N$ such that $\|A-B\|_{\mathfrak{L}_1}\leq \frac{\varepsilon}{2D}$. Let $a\in \A_{\infty^d,\sigma}$ with $\Lip_{l,\infty^d,\sigma}(a)\leq 1$. By \cite{Rieffel99}, there exists $t\in\R$ such that $\|a+t\unit\|_{\infty^d,\sigma}\leq D$. Then, 
\begin{equation}\label{qt-states-eq2}
\begin{split}
|\psi_A\circ\pi_{\infty^d,\sigma}(a) - \psi_B\circ\pi_{\infty^d,\sigma}(a)|&=|\psi_A\circ\pi_{\infty^d,\sigma}(a+t\unit) - \psi_B\circ\pi_{\infty^d,\sigma}(a+t\unit)|\\
&=|\tr( (A-B)\pi_{\infty^d,\sigma}(a+t\unit))|\\
&\leq \|(A-B)\pi_{\infty^d,\sigma}(a+t\unit)\|_{\mathfrak{L}_1}\\
&\leq \|A-B\|_{\mathfrak{L}_1}\|a+t\unit\|_{\infty^d,\sigma} \\
&\leq \frac{\varepsilon}{2D}D = \frac{\varepsilon}{2}\text{.}
\end{split}
\end{equation}

Thus, we conclude from Inequalities (\ref{qt-states-eq1}) and (\ref{qt-states-eq2}):

\begin{equation*}
\Haus{\Kantorovich{\Lip_{l,\infty^d,\sigma}}}(\StateSpace(\A_{\infty^d,\sigma}),\{ \psi_A\circ\pi_{\infty^d,\sigma} : A\in \mathfrak{F}_N \}) \leq \varepsilon\text{,}
\end{equation*}
as desired.
\end{proof}

%%%%%%%%%%%%%%%%%%

Our proof of Theorem (\ref{qt-main}) to follow relies, in part, on the notion of continuous field of states and their relation with continuous fields of Lip-norms, as in \cite{Rieffel00}. We propose to recall the main tools from \cite{Rieffel00} we shall need now, to enhance the clarity of our exposition.

As in \cite{Rieffel00}, we assume given a finite dimensional vector space $V$ endowed with a family $(\|\cdot\|_\theta)_{\theta\in\Theta}$ of norms indexed by some compact metric space $\Theta$, such that for all $v \in V$, the map $\theta\in\Theta\mapsto \|v\|_\theta$ is continuous. Using the compactness of $\Theta$, one can then find a norm $\|\cdot\|_\ast$ on $V$ which dominates $\|\cdot\|_\theta$ for all $\theta\in\Theta$; then by Lemma (\ref{cont-lemma}), the map $(v,\theta)\in V\times\Theta\mapsto \|v\|_\theta$ is jointly continuous. Moreover, by \cite[Lemma 10.1]{Rieffel00}, the family $(\|\cdot\|'_\theta)_{\theta\in\Theta}$ of dual norms is continuous as well on the dual $V'$ of $V$.

Let $e$ be an nonzero element of $V$. Assume now that for all $\theta \in \Theta$, the norm $\|\cdot\|_\theta$ is associated to some order-unit space structure on $V$ with order-unit $e$, and let $\mathscr{S}_\theta$ be the state space for this order-unit space structure. A \emph{continuous field of states} over $\Theta$ is a function $\theta\in\Theta\mapsto \varphi_\theta \in V'$ such that $\varphi_\theta\in\mathscr{S}_\theta$ for all $\theta\in\Theta$ and $\theta\in\Theta\mapsto\varphi_\theta(v)$ is continuous for all $v\in V$. Rieffel showed in \cite[Proposition 10.9,10.10]{Rieffel00} that in this context, there exists many continuous fields of states. For our purpose, however, this will be obtained in a different manner in the proof of Theorem (\ref{qt-main}).

The main result for us is the following:

\begin{theorem}[Rieffel, Lemma 10.11 in \cite{Rieffel00}]\label{rieffel-cont-field-thm}
Let $V$ be a finite dimensional space, $e\in V$, and $\Theta$ be a compact metric space such that for all $\theta\in\Theta$, we are given an order-unit space structure on $V$ of order unit $e$, with state space $\mathscr{S}_\theta$ and norm $\|\cdot\|_\theta$. Assume moreover that $\theta\in\Theta\mapsto \|v\|_\theta$ is continuous for all $v\in V$. Then:
\begin{enumerate}
\item There exists a norm $\|\cdot\|^\ast$ on the dual $V'$ of $V$ and $k \geq 1$ such that:
\begin{equation*}
\|\cdot\|^\ast \leq \|\cdot\|_\theta'\leq k \|\cdot\|^\ast
\end{equation*}
for all $\theta\in\Theta$,
\item If $\varepsilon > 0$ is given, and $\mathscr{P}$ is a set of continuous families of states such that, for some $\omega\in\Theta$, the set $\{ \varphi_\omega : \varphi \in \mathscr{P} \}$ is $\varepsilon$-dense in $\mathscr{S}_\omega$ for $\|\cdot\|^\ast$, then there exists a neighborhood $U$ of $\omega$ in $\Theta$ such that for all $\theta\in U$, the set $\{\varphi_\theta : \varphi \in \mathscr{P} \}$ is $3\varepsilon$-dense in $\mathscr{S}_\theta$ for $\|\cdot\|^\ast$.
\end{enumerate}
\end{theorem}

We refer to \cite[Section 10]{Rieffel00} for the development of the theory which leads to the proof of Theorem (\ref{rieffel-cont-field-thm}). Note that our $k$ is $k^{-1}$ in the notations of \cite[Lemma 10.1]{Rieffel00}.

%%%%%%%%%%%%%%%%%%

We now establish the fundamental example of this paper:

\begin{theorem}\label{qt-main}
Let $d\in \N\setminus\{0,1\}$ and $\sigma$ a skew-bicharacter of $\Z^d$. Write $\infty^d = (\infty,\ldots,\infty) \in \Nbar^d$. Let $l$ be a continuous length function on $\U^d$. Then:
\begin{equation*}
\lim_{(c,\theta)\rightarrow (\infty^d,\sigma)} \propinquity\left(\left(C^\ast\left(\Z^d_c,\theta\right),\Lip_{l,c,\theta}\right),\left(C^\ast\left(\Z^d,\sigma\right),\Lip_{l,\infty^d,\sigma}\right)\right) = 0 \text{.}
\end{equation*}
\end{theorem}

\begin{proof}

Our proof consists of four steps, which we separate in claims, followed by their own proofs. The strategy consists of finding, for any $\varepsilon > 0$, a neighborhood $\Omega$ of $(\infty^d,\sigma)$ in $\Xi^d$ such that, for all $(c,\theta)\in \Omega$, there exists a bridge $\gamma^\varepsilon_{c,\theta}$ in $\bridgeset{\A_{\infty^d,\sigma}}{\A_{c,\theta}}$ with height less or equal than $\varepsilon$. Then, we check that, up to shrinking $\Omega$, these bridge have reach less than $\varepsilon$ as well. We start by establishing the framework for our proof.

\bigskip

For the rest of this proof, we fix $\varepsilon > 0$. 

\bigskip

By Theorem (\ref{Fejer-approx-thm}), there exists $\phi : \U^d \rightarrow \R$ and an open neighborhood $\Omega_0\subseteq\Xi^d$ of $({\infty^d,\sigma})$ in $\Xi^d$ such that:
\begin{enumerate}
\item For all $(c,\theta) \in \Omega_0$ and for all $a\in \sa{C^\ast(\Z^d_c,\theta)}$ we have:
\begin{equation*}
\|a-\alpha_{c,\theta}^\phi (a)\|_{c,\theta} \leq \frac{1}{4}\varepsilon \Lip_{l,c,\theta}(a)\text{ and }\Lip_{l,c,\theta}(\alpha_{c,\theta}^\phi(a))\leq\Lip_{l,c,\theta}(a)\text{,}
\end{equation*}
\item There exists a finite subset $S$ of $\Z^d$ containing $0$ such that, for all $(c,\theta) \in \Omega_0$, the restriction of the canonical surjection $q_c: \Z^d \rightarrow \Z^d_c$ is injective on $S$ while the range of $\alpha_{c,\theta}^\phi$ is the span of $\{ U_{c,\theta}^p : p \in q_c(S) \}$, where the unitaries $U_{c,\theta}^p$ ($p \in \Z^d_c$) are defined by Equation (\ref{U-eq}) in Theorem (\ref{representation-thm}).
\end{enumerate}

Since $\Xi^d$ is compact, we assume that $\Omega_0$ is chosen to be a compact neighborhood of $({\infty^d},\sigma)$, shrinking it if necessary.

We define the following two vector spaces, which we will use repeatedly in the rest of this proof:
\begin{equation*}
V = \left\{ f \in \ell^1(\Z^d) : \forall n \not\in S \quad f(n) = 0 \right\}\text{,}
\end{equation*}
and
\begin{equation*}
E = \{ f \in V : f(0) = 0 \} \text{.}
\end{equation*}

We also denote the unit ball of $(E,\|\cdot\|_1)$ by $\Sigma$:
\begin{equation}\label{sigma-def-eq}
\Sigma = \left\{ f \in E : \|f\|_1 = 1 \right\}\text{.}
\end{equation}

By construction, $V$ is a finite dimensional subspace of $\A_{\infty^d,\sigma}$. Moreover, since $q_c$ is injective on $S$, the map $\upsilon_c$ is injective on $V$ by Lemma (\ref{upsilon-lemma}); thus it defines a linear isomorphism from $V$ onto:
\begin{equation*}
\upsilon_c(V) = \left\{ f \in \ell^1\left(\Z^d_c\right) : \forall n \not\in S \quad f(q_c(n)) = 0 \right\}\text{.}
\end{equation*}
For all $(c,\theta)\in \Omega_0$, the map $\upsilon_c$ thus allows us to identify $V$ (and $E$) with a subset of $\A_{c,\theta}$, and to drop the notation $\upsilon_c$ with no confusion. We shall thus do so in the rest of this proof.

\bigskip

\emph{The space $V$, identified with a subspace of $\A_{c,\theta}$ for all $(c,\theta)\in\Omega_0$, carries a structure of order-unit space.}

For all $(c,\theta)\in \Omega_0$, the space $V$ is identified with a subspace of $\sa{\A_{c,\theta}}$ containing the unit, and thus $V$ is endowed with an order-unit space structure associated with the norm $\|\cdot\|_{c,\theta}$; let $\StateSpace(V|c,\theta)$ be the state space of $V$ with this order-unit space structure. Note that the unit is always $\delta_0$ for all $(c,\theta)\in \Omega_0$, and thus we will omit it from our notations.

By Lemma (\ref{cont-lemma}), for any $v\in V$, the function $(c,\theta)\in \Omega_0 \mapsto \|v\|_{c,\theta}$ is continuous (using $m = \|\cdot\|_1$, in the notations of Lemma (\ref{cont-lemma})). 

\bigskip
\emph{We now check that $V \subseteq \dom{\Lip_{l,c,\theta}}$ for all $(c,\theta)\in\Omega_0$.}

For all $(c,\theta)\in\Omega_0$, we note that the dual action $\alpha_{c,\theta}$ of $\U^d_c$ on $\A_{c,\theta}$ given by Theorem-Definition (\ref{dual-action-thmdef}) leaves $V$ invariant. By \cite[Proposition 2.2]{Rieffel98a}, the vector space $V_\Lip$ of elements $v$ in $V$ with $\Lip_{l,c,\theta}(v)<\infty$ is dense in $V$, and since $V_\Lip$ is finite dimensional (as $V$ is), it is closed in $V$. We thus conclude that $V\subseteq\dom{\Lip_{l,c,\theta}}$ for all $(c,\theta) \in \Omega_0$.

\emph{Last, we introduce a quantity we shall use several times in the following claims for our proof.} By Corollary (\ref{norm-cont-corollary}) and Theorem (\ref{lip-norm-cont-thm}), as $E$ is a finite dimensional subspace of $\ell^1(\Z^d)$, the functions $\mathrm{lf}: (f,c,\theta)\in \Sigma\times \Omega \mapsto\Lip_{l,c,\theta}(f)$ is jointly continuous, where $\Sigma$ is defined by Equation (\ref{sigma-def-eq}). Since $\Sigma\times \Omega$ is compact, the function $\mathrm{lf}$ has a minimum on $\Sigma\times \Omega$. We set:
\begin{equation}\label{qt-y-eq}
y = \min \left\{ \Lip_{l,c,\theta}(f) : (f,c,\theta)\in\Sigma\times\Omega_0 \right\} \text{.}
\end{equation}
By assumptions on Lip-norms and since $\Sigma$ contains no nonzero scalar multiple of the unit of any $\A_{c,\theta}$ for any $(c,\theta)\in\Xi^d$, the real number $y$ is strictly positive.

We can now start our series of claims, based on the framework established above.

\begin{claim}\label{continuous-fields-of-states-claim}
We shall use the notation of Lemma (\ref{F_N-lemma}): for any $A\in\mathfrak{L}_1^+$, the state $\psi_A$ on $\B^d$ is defined as:
\begin{equation*}\psi_A : T\in\B^d \longmapsto \tr(AT)\text{.}\end{equation*}
There exists a compact neighborhood $\Omega$ of $(\infty^d,\sigma)$ in $\Xi^d$, $N\in\N$ and a finite subset $\mathfrak{F}_N$ of $\mathfrak{L}_1^+$ such that:
\begin{enumerate}
\item $\forall (c,\theta) \in \Omega\quad \Haus{\Kantorovich{\Lip_{l,c,\theta}}}\left(\StateSpace(\A_{c,\theta}),\left\{ \psi_A\circ\pi_{c,\theta} : A \in \mathfrak{F}_N \right\} \right) < \varepsilon \text{.}$
\item For all $B \in \mathfrak{F}_N$, we have $P_N B = B P_N = P_N B P_N = B$, where $P_N$ is the projection of $\ell^2\left(\Z^d\right)$ on the span of $\{e_n : |n|\leq N\}$, with $(e_n)_{n\in\Z^d}$ the canonical Hilbert basis of $\ell^2\left(\Z^d\right)$.
\end{enumerate}
\end{claim}

\newcommand{\lip}{{\mathsf{l}}}
Let $\lip_{c,\theta}$ be the restriction of $\Lip_{l,c,\theta}$ to $E$ for all $(c,\theta)\in\Omega_0$. Since $\Lip_{l,c,\theta}(v) = 0$ implies that $v\in\R\delta_0$ for any $v\in V$ by definition of Lip-norms, and since we saw that $V\subseteq\dom{\Lip_{l,c,\theta}}$, we conclude that $\lip_{c,\theta}$ is a norm on $E$. Let $\lip^\ast_{c,\theta}$ be the dual norm of $\lip_{c,\theta}$, defined on the dual $E^\ast$ of $E$, i.e.:
\begin{equation*}
\forall\mu\in E^\ast\quad\lip^\ast_{c,\theta}(\mu) = \sup \{ |\mu(v)| : v\in E\text{ and }\Lip_{l,c,\theta}(v) \leq 1\}\text{.}
\end{equation*}
As shown in \cite[Lemma 10.1]{Rieffel00} and as restated in Theorem (\ref{rieffel-cont-field-thm}), there exists a norm $\|\cdot\|^\ast$ on the dual $E^\ast$ of $E$ and a constant $k \geq 1$ such that, for all $(c,\theta) \in \Omega_0$, we have:
\begin{equation}\label{norm-equiv-eq}
\|\cdot\|^\ast\leq \lip_{c,\theta}^\ast(\cdot)\leq k \|\cdot\|^\ast\text{.}
\end{equation}
One also checks that for all $v \in E$, the map $(c,\theta)\in\Omega_0\mapsto \lip^\ast_{c,\theta}(v)$ is continuous by Lemma (\ref{cont-lemma}), since $(c,\theta,v)\in\Omega\times E\mapsto\lip_{c,\theta}(v)$ is jointly continuous.

It is important to take note that by Definition (\ref{quantum-compact-metric-space-def}), we have:
\begin{equation*}
\Kantorovich{\Lip_{l,c,\theta}}(\mu,\nu) \geq \lip^\ast_{c,\theta}(\mu-\nu)
\end{equation*}
for any $(c,\theta)\in \Omega_0$ and $\mu,\nu\in\StateSpace(\A_{c,\theta})$, where we denote the restrictions of $\mu$ and $\nu$ by the same letters.
\bigskip

\emph{We now can use Lemma (\ref{F_N-lemma}) and \cite[Lemma 10.11]{Rieffel00} to construct our continuous field of states.}

Let $k\geq 1$ be given by Assertion (\ref{norm-equiv-eq}). By Lemma (\ref{F_N-lemma}), there exists $N\in\N$ and a finite subset $\mathfrak{F}_N$ of $\mathfrak{L}_1^+$ such that:
\begin{equation*}
\Haus{\Kantorovich{\Lip_{l,\infty^d,\sigma}}}(\StateSpace(\A_{\infty^d,\sigma}),\{ \psi_A\circ\pi_{\infty^d,\sigma} : A\in \mathfrak{F}_N \}) \leq \frac{\varepsilon}{6k}\text{,}
\end{equation*}
and such that $P_N A P_N = A P_N = P_N A = A$ for all $A\in\mathfrak{F}_N$ with $P_N$ be the orthogonal projection on $\operatorname{span}\{e_j : |j| \leq N \}$.

In particular, for any $\varphi \in\StateSpace(\A_{\infty^d,\sigma})$, there exists $A\in \mathfrak{F}_N$ such that:
\begin{equation}\label{star-norm-eq}
\begin{split}
\| \varphi - \psi_\A\circ\pi_{\infty^d,\sigma} \|^\ast &\leq \lip^\ast_{\infty^d,\sigma}(\varphi - \psi_\A\circ\pi_{\infty^d,\sigma}) \\
&\leq \Kantorovich{\Lip_{l,\infty^d,\sigma}}(\varphi,\psi\circ\pi_{\infty^d,\sigma}) \leq \frac{\varepsilon}{6k}\text{,}
\end{split}
\end{equation}
where we used the same names for $\varphi$ and $\psi_\A\circ\pi_{k,\sigma}$ and their restrictions to $E$.

\bigskip

\emph{We now use \cite{Rieffel00} to build our continuous fields of states.}

For each $A\in{\mathfrak{L}_1}^{1+}$ and $v\in V$, the function:
\begin{equation*}
(c,\theta) \in \Omega_0 \longmapsto \psi_A\circ\pi_{c,\theta}(v)
\end{equation*}
is continuous, by Lemma (\ref{trace-class-SOT-lemma}) and Theorem (\ref{SOT-cont-thm}). 

By \cite[Lemma 10.11]{Rieffel00}, there exists a compact neighborhood $\Omega \subseteq \Omega_0$ of $(\infty^d,\sigma)$ such that, for all $(c,\sigma) \in \Omega$, and for all $\varphi \in\StateSpace(V|c,\theta)$, there exists $A\in\mathfrak{F}_N$ such that:
\begin{equation*}
\|\varphi - \psi_A\circ\pi_{c,\theta}\|^\ast \leq \frac{3\varepsilon}{6 k} = \frac{\varepsilon}{2k}\text{,}
\end{equation*}
where, by a slight abuse of notations, we identified the states $\psi_A\circ\pi_{c,\theta}$ and $\varphi$ with their restrictions to $E$. 

\bigskip

\emph{We can now conclude our claim.}

Let $(c,\theta)\in\Omega$. Since $\lip_{c,\theta}^\ast(\cdot) \leq k \|\cdot\|^\ast$ on $E^\ast$, and since all states map the unit to $1$, we conclude that, for all $\varphi\in\StateSpace(V|c,\theta)$, there exists $A\in\mathfrak{F}_N$ such that:
\begin{equation}\label{qt-height-eq-1}
\sup \{ | \varphi(a) - \psi_A\circ\pi_{c,\theta}(a) | : a \in V \text{ and }\Lip_{l,c,\theta}(a) \leq 1\} \leq k \frac{\varepsilon}{2k} = \frac{\varepsilon}{2}\text{.}
\end{equation}
The expression in the left hand-side of Inequality (\ref{qt-height-eq-1}) is the {\mongekant} on $\StateSpace(V|c,\theta)$ associated with the restriction of $\Lip_{l,c,\theta}$ to $V$; we shall not need to worry about introducing a notation for this.

\bigskip
Let $\varphi \in \StateSpace(\A_{c,\theta})$. There exists $A\in \mathfrak{F}_N$ such that Inequality (\ref{qt-height-eq-1}) holds. Now, for any $a\in \sa{\A_{c,\theta}}$ with $\Lip_{l,c,\theta}(a) \leq 1$, we have:
\begin{equation*}
\begin{split}
|\varphi(a) - \psi_A\circ\pi_{c,\theta}(a)| &\leq | \varphi(a) - \varphi(\alpha_{c,\theta}^\phi(a))| + |\varphi\circ\alpha_{c,\theta}^\phi(a) - \psi_A\circ\pi_{c,\theta}(\alpha_{c,\theta}^{\phi}(a))|\\
&\quad +| \psi_A\circ\pi_{c,\theta}(\alpha_{c,\theta}^{\phi}(a)) -  \psi_A\circ\pi_{c,\theta}(a)|\\
&\leq \frac{\varepsilon}{4} + \frac{\varepsilon}{2} + \frac{\varepsilon}{4}\text{,}
\end{split}
\end{equation*} 
since $\Lip_{l,c,\theta}(\alpha_{c,\theta}^\phi(a)) \leq \Lip_{l,c,\theta}(a)$ and $\alpha_{c,\theta}^\phi(a) \in V$.

\bigskip
So we conclude that for all $(c,\theta)\in \Omega_1$, we have:
\begin{equation}\label{qt-height-eq0}
\Haus{\Kantorovich{\Lip_{l,c,\theta}}}(\StateSpace(\A_{c,\theta}),\{ \psi_A\circ\pi_{c,\theta} : A\in \mathfrak{F}_N \}) \leq \varepsilon\text{,}
\end{equation}
as desired.

\bigskip

\begin{claim}\label{bridge-height-claim}
Let:
\begin{equation}\label{main-K-eq}
K = \max\{ |m| : m \in S \}\text{,}
\end{equation}
where $S$ is the support of the Fourier transform of $\phi$ on $\U$ (which is finite by assumption). Let $\Omega$ be the compact neighborhood of $(\infty^d,\sigma)$ given by Claim (\ref{continuous-fields-of-states-claim}).

There exist $N,M\in\N$ such that, for all $(c,\theta) \in \Omega$, the bridge:
\begin{equation*}
\gamma^\varepsilon_{c,\theta} = \left( \B\left(\ell^2\left(\Z^d\right)\right), \omega_{N,M}, \pi_{\infty^d,\sigma},\pi_{c,\theta} \right)
\end{equation*}
in $\bridgeset{\A_{\infty^d,\sigma}}{\A_{c,\theta}}$, where $\omega_{N,M}$ is given in Theorem (\ref{qt-commutator-thm}), has height less or equal than $\varepsilon$, and moreover:
\begin{equation*}
\frac{K}{M} \leq \frac{\varepsilon y }{4} \text{.}
\end{equation*}
\end{claim}

Let $M \in \N$ be chosen so that $\frac{K}{M} < \frac{\varepsilon y}{4}$, with $y$ defined by Equation (\ref{qt-y-eq}) and $K$ defined by Equation (\ref{main-K-eq}). 

 We define, using the notations of Theorem (\ref{qt-commutator-thm}) and Notation (\ref{weight-notation}):
\begin{equation*}
\omega_{N,M} = \Diag{w_{N,M} (n)}{n \in \Z^d }\text{,}
\end{equation*}
where $N$ is given by Claim (\ref{continuous-fields-of-states-claim}). In particular, using the notations and conclusions of Claim (\ref{continuous-fields-of-states-claim}), $P_N \omega_{N,M} = \omega_{N,M}P_N = P_N$ and $\omega_{N,M}$ is a positive trace class operator in $\B^d$.

Therefore, by construction:
\begin{equation*}
\psi_A((\unit_{\B^d} - \omega_{N,M})^\ast(\unit_{\B^d} - \omega_{N,M})) = \psi_A((\unit_{\B^d}-\omega_{N,M})(\unit_{\B^d}-\omega_{N,M})^\ast) = 0
\end{equation*}
for all $A\in\mathfrak{F}_N$. Hence:
\begin{equation*}
\left\{ \psi_A :  A\in\mathfrak{F}_N\right\} \subseteq \mathscr{S}_1(\omega_{N,M})\text{.}
\end{equation*} 
Hence, by Inequality (\ref{qt-height-eq0}), we have for all $(c,\theta)\in\Omega$:
\begin{equation}\label{qt-height-eq0b}
\Haus{\Kantorovich{\Lip_{l,c,\theta}}}\left(\StateSpace(\A_{c,\theta}),\left\{ \psi\circ\pi_{c,\theta} : \psi\in\mathscr{S}_1\left(\omega_{N,M}\right)\right\}\right) \leq \varepsilon\text{.}
\end{equation}
Thus, by Inequality (\ref{qt-height-eq0b}) and Definition (\ref{height-def}), we have:
\begin{equation}\label{qt-height-eq1}
\bridgeheight{\gamma}{\Lip_{l,\infty^d,\sigma},\Lip_{l,c,\theta}} \leq \varepsilon
\end{equation}
for all $(c,\theta)\in\Omega$.

\bigskip

\begin{claim}\label{bridge-length-claim}
There exists a neighborhood $\Omega'$ of $(\infty^d,\sigma)$ in $\Xi^d$ and $N,M \in \N$ such that, for all $(c,\theta)\in\Omega'$, the length of the bridge:
\begin{equation*}
\left( \B\left(\ell^2\left(\Z^d\right)\right), \omega_{N,M}, \pi_{\infty^d,\sigma},\pi_{c,\theta} \right)
\end{equation*}
is less than $\varepsilon$.
\end{claim}

Using Notation (\ref{wedge-notation}), let:
\begin{equation*}
\Omega_> = \Omega \cap \left\{ z \in \Nbar^d :  N+M \leq \wedge z \right\}\times\mathds{B}_{\infty^d}
\end{equation*}
which is a compact neighborhood of $(\infty^d,\sigma)$ by definition. 

\bigskip

Let $N,M$ be given by Claim (\ref{bridge-height-claim}). By Theorem (\ref{lip-norm-cont-thm}), together with Lemma (\ref{cont-lemma}), choosing $\|\cdot\|_1$ for $\mathsf{m}$, we see that the map:
\begin{equation*}
\mathrm{df} : (f,c,\theta)\in \Sigma\times \Omega_> \longmapsto |\Lip_{c,\theta}(f) - \Lip_{\infty^d,\sigma}(f)|
\end{equation*}
is jointly continuous. Similarly, by Corollary (\ref{compact-sot-corollary}), and using the same method as used in the proof of Lemma (\ref{cont-lemma}), the map:
\begin{equation*}
\mathrm{cf}: (f,c,\theta)\in V \times \Omega_> \longmapsto \|(\pi_{\infty^d,\sigma}(\upsilon_k(f)) - \pi_{c,\theta}(\upsilon_c(f)))\omega_{N,M}\|_{\B^d}
\end{equation*}
is jointly continuous. Thus $\mathrm{cf}$ and $\mathrm{df}$ are uniformly continuous on the compact $\Sigma\times \Omega_>$. We can therefore find a compact neighborhood $\Omega'$ of $(\infty^d,\sigma)$ in $\Omega_>$ such that, for all $(c,\theta) \in \Omega'$, and for all $f \in \Sigma$, we have:
\begin{equation*}
\left\{
\begin{aligned}
&\|(\pi_{\infty^d,\sigma}(\upsilon_k(f)) - \pi_{c,\theta}(\upsilon_c(f)))\omega_{N,M}\|_{\B^d} < \frac{1}{4}\varepsilon y\text{,} \\
&|\Lip_{\infty^d,\sigma}(f) - \Lip_{c,\theta}(f)| < \frac{1}{4}\varepsilon y^2 \text{.}
\end{aligned}
\right.
\end{equation*}

\bigskip

Let $(c,\theta) \in \Omega'$. The bridge we wish to consider is given by:
\begin{equation*}
\gamma = \left( \B(\ell^2(\Z^d)), \omega_{N,M}, \pi_{\infty^d,\sigma}, \pi_{c,\theta}  \right) \text{.}
\end{equation*}
By Claim (\ref{bridge-height-claim}), the height of $\gamma$ is no more than $\varepsilon$.
\bigskip

It remains to compute the reach of the bridge $\gamma$. 

\bigskip

Let $a\in \Sigma \subseteq \sa{\A_{\infty^d,\sigma}}$, and note that $\Lip_{l,c,\theta}(a) > 0$ by definition of $\Sigma$. Set:
\begin{equation*}
b = \frac{\Lip_{l,\infty^d,\sigma}(a)}{\Lip_{l,c,\theta}(a)} a \text{.}
\end{equation*}
By construction, $\Lip_{l,c,\theta}(b) = \Lip_{l,\infty^d,\sigma}(a)$. 

Now, using Theorem (\ref{qt-commutator-thm}), since for all $(c,\theta)\in\Omega'$ we have $N+M\leq \wedge c$, we have for all $d = \sum_{n\in S} \lambda_n \delta_n \in\Sigma$, with $(\lambda_n)_{n\in S}$ a family of scalar:
\begin{equation*}
\begin{split}
\|[\omega_{N,M},\pi_{c,\theta}(d)]\|_{\B^d} &\leq \sum_{m\in S} |\lambda_m| \|[\omega_{N,M},\pi_{c,\theta}(\delta_{m})\|_{\B^d} \\
&\leq \sum_{m\in S} |\lambda_m| \frac{K}{M} \leq \frac{K}{M} \leq\frac{\varepsilon y}{4}\text{,}
\end{split}
\end{equation*}
since, again by definition of $\Sigma$ as the unit sphere in $E$ for $\|\cdot\|_1$, we have:
\begin{equation*}
\sum_{m\in S}|\lambda_m| = 1\text{.}
\end{equation*}

Using $\|\omega_{N,M}\|_{\B^d} = 1$, we thus have:
\begin{multline*}
\| \pi_{\infty^d,\sigma}(a)\omega_{N,M} - \omega_{N,M}\pi_{c,\theta}(b) \|_{\B^d} \\
\begin{aligned}
&\leq \| [\omega_{N,M},\pi_{c,\theta}(b)] \|_{\B^d} + \| (\pi_{\infty^d,\sigma}(a) - \pi_{c,\theta}(b)) \omega_{N,M} \|_{\B^d}\\
&\leq \frac{\varepsilon y}{4} +  \|(\pi_{\infty^d,\sigma}(a) - \pi_{c,\theta}(a))\omega_{N,M}\|_{\B^d} + \| (\pi_{c,\theta}(a) - \pi_{c,\theta}(b)) \|_{\B^d}\\
&\leq \frac{\varepsilon}{4} y + \frac{\varepsilon}{4} y + \left|1-\frac{\Lip_{l,\infty^d,\sigma}(a)}{\Lip_{l,c,\theta}(a)}\right|\\
&\leq \frac{1}{2}\varepsilon\Lip_{l,\infty^d,\sigma}(a) + \frac{|\Lip_{l,\infty^d,\sigma}(a)-\Lip_{l,c,\theta}(a)|}{\Lip_{l,c,\theta}(a)}\\
&\leq \frac{\varepsilon}{2}\Lip_{l,\infty^d,\sigma}(a) + \frac{\varepsilon y}{4}\leq \frac{3\varepsilon}{4}\Lip_{l,\infty^d,\sigma}(a) \text{.}
\end{aligned}
\end{multline*}

Now, let $a = a' + t\unit_{\A_{\infty^d,\sigma}} \in V$ with $a' \in \Sigma$. A quick computation shows that, if $b' = \frac{\Lip_{l,\infty^d,\sigma}(a')}{\Lip_{l,c,\theta}(a')}a'$  and $b = b' + t\unit_{\A_{c,\theta}}$, then:
\begin{equation}\label{qt-final-construction-eq1}
\Lip_{l,c,\theta}(b) \leq \Lip_{l,\infty^d,\sigma}(a) \text{ and }\|\pi_{\infty^d,\sigma}(a) \omega_{N,M} - \omega_{N,M}\pi_{c,\theta}(b)\|_{\B^d} \leq \frac{3\varepsilon}{4}\Lip_{l,\infty^d,\sigma}(a) \text{.}
\end{equation}

Now, let $d\in V$ and let $a\in E$ with $d-a\in\R\delta_0$. Let $r = \|a\|_1$. If $r = 0$, then $d = t\unit_{\A_{\infty^d,\sigma}}$ for some $t\in\R$, and we check easily that if $e = t\unit_{\A_{c,\theta}}$ then we have $\Lip_{l,\infty^d,\sigma}(d) = \Lip_{l,c,\theta}(e) (=0)$ and $\|\pi_{\infty^d,\sigma}(d)\omega_{N,M} - \omega_{N,M}\pi_{c,\theta}(e)\|_{\B^d} = 0$.

Otherwise, pick $b \in \sa{\A_{c,\theta}}$ given by Equation (\ref{qt-final-construction-eq1}) for $a = r^{-1}d$, and set $e = rc$. Then again, by homogeneity:
\begin{equation}\label{qt-final-construction-eq2}
\Lip_{l,c,\theta}(e) \leq \Lip_{l,\infty^d,\sigma}(d) \text{ and }\|\pi_{\infty^d,\sigma}(d) \omega_{N,M} - \omega_{N,M}\pi_{c,\theta}(e)\|_{\B^d} \leq \frac{3\varepsilon}{4}\Lip_{l,\infty^d,\sigma}(d) \text{.}
\end{equation}

Last, let $a\in \A_{\infty^d,\sigma}$. let $a' = \alpha_{\infty^d,\sigma}^\phi(a)$. Since $a' \in V$, there exists, by Equation (\ref{qt-final-construction-eq2}), an element $b \in \sa{\A_{c,\theta}}$ such that:
\begin{equation}\label{qt-final-construction-eq3}
\begin{cases}
\Lip_{l,c,\theta}(b)\leq \Lip_{l,\infty^d,\sigma}(a')\leq \Lip_{l,\infty^d,\sigma}(a)\text{ and }\\
\|\pi_{\infty^d,\sigma}(a')\omega_{N,M} - \omega_{N,M}\pi_{c,\theta}(b)\|_{\B^d} \leq \frac{3\varepsilon}{4}\Lip_{l,\infty^d,\sigma}(a')\text{.}
\end{cases}
\end{equation}

 Then we have (again, using $\|\omega_{N,M}\|_{\B^d} = 1$):
\begin{multline*}
\|\pi_{\infty^d,\sigma}(a) \omega_{N,M} - \omega_{N,M} \pi_{c,\theta}(b)\|_{\B^d}\\
\begin{aligned}
&\leq \|\left(\pi_{\infty^d,\sigma}(a) - \pi_{\infty^d,\sigma}(a')\right)\omega_{N,M}\|_{\B^d} + \|\pi_{\infty^d,\sigma}(a')\omega_{N,M} - \omega_{N,M} \pi_{c,\theta}(b)\|_{\B^d} \\
&\leq \|\left(\pi_{\infty^d,\sigma}(a) - \pi_{\infty^d,\sigma}(a')\right)\|_{\B^d} + \varepsilon\Lip_{l,\infty^d,\sigma}(a)\text{ as $\|\omega
_{N,M}\|_{\B^d} = 1$,}\\
&\leq \frac{\varepsilon}{4}\Lip_{\infty^d,\sigma}(a) + \frac{3\varepsilon}{4} \Lip_{k,\sigma}(a)\leq \varepsilon\Lip_{k,\sigma}(a) \text{.}
\end{aligned}
\end{multline*}

\bigskip

In conclusion, for all $a\in\dom{\Lip_{l,\infty^d,\sigma}}$, there exists $b\in \sa{\A_{c,\theta}}$ such that $\Lip_{l,c,\theta}(b) \leq \Lip_{l,\infty^d,\sigma}(a)$ and:
\begin{equation*}
\|\pi_{\infty^d,\sigma}(a)\omega_{N,M} - \omega_{N,M}\pi_{c,\theta}(b)\|_{\B^d} \leq \varepsilon \Lip_{l,\infty^d,\sigma}(a)\text{.}
\end{equation*}

The above reasoning applies equally well with the roles of $(\infty^d,\sigma)$ and $(c,\theta)$ reversed. This proves that:
\begin{equation*}
\bridgereach{\gamma}{\Lip_{l,\infty^d,\sigma},\Lip_{l,c,\theta}} \leq \varepsilon\text{.}
\end{equation*}

Together with Equation (\ref{qt-height-eq1}), we have thus, by Definition (\ref{length-def}), that:
\begin{equation*}
\bridgelength{\gamma}{\Lip_{l,\infty^d,\sigma},\Lip_{l,c,\theta}} \leq \varepsilon\text{.}
\end{equation*}

\begin{claim}
There exists a neighborhood $\Omega'$ of $(\infty^d,\sigma)$ in $\Xi^d$ such that for all $(c,\theta) \in \Omega'$, we have:
\begin{equation*}
\propinquity((\A_{\infty^d,\sigma},\Lip_{l,\infty^d,\sigma}),(\A_{c,\theta},\Lip_{l,c,\theta})) \leq \varepsilon \text{.}
\end{equation*}
\end{claim}

Let $\Omega'$,$N,M\in\N$ be given by Claim (\ref{bridge-length-claim}). Let $(c,\theta)\in\Omega'$. Set:
\begin{equation*}
\gamma = \left(\B\left(\ell^2(\Z^d)\right),\omega_{N,M},\pi_{\infty^d,\sigma},\pi_{c,\theta}\right)
\end{equation*} 
If $\Gamma = (\A_{\infty^d,\sigma},\Lip_{k,\sigma},\gamma,\A_{c,\theta},\Lip_{l,c,\theta})$ is the canonical trek associated with $\gamma$, we have $\treklength{\Gamma} = \bridgelength{\gamma}{\Lip_{l,\infty^d,\sigma},\Lip_{l,c,\theta}} \leq \varepsilon$, and thus by Definition (\ref{propinquity-def}) we have:
\begin{equation*}
\propinquity((\A_{\infty^d,\sigma},\Lip_{l,\infty^d,\sigma}),(\A_{c,\theta},\Lip_{l,c,\theta})) \leq \varepsilon \text{,}
\end{equation*}
as desired.
\end{proof}

\begin{remark}
As the construction of Theorem (\ref{qt-main}) involves a trek with a single bridge, Theorem (\ref{qt-main}) also holds if we replace the quantum propinquity with Rieffel's proximity \cite{Rieffel10c} or with the quantum propinquity specialized to compact C*-metric spaces (or any class of {\Lqcms s} which contain the quantum tori and the fuzzy tori). Thus in particular:
\begin{equation*}
\lim_{(c,\theta)\rightarrow (\infty^d,\sigma)} \propinquity_{\mathcal{C}^\ast}\left(\left(C^\ast\left(\Z^d_c,\theta\right),\Lip_{l,c,\theta}\right),\left(C^\ast\left(\Z^d,\sigma\right),\Lip_{l,\infty^d,\sigma}\right)\right) = 0 \text{.}
\end{equation*}
\end{remark}

Thus, the quantum Gro\-mov-Haus\-dorff propinquity allows one to talk about convergence of fuzzy tori to quantum tori in a stronger sense that in \cite{Latremoliere05}; in particular the convergence preserves the C*-algebra structure, per \cite[Theorem 5.13]{Latremoliere13}.

We conclude by a remark about continuity for fuzzy tori for the quantum propinquity. The proof of Theorem (\ref{qt-main}) can be adjusted to show that:
\begin{equation*}
\lim_{ (c,\theta)\rightarrow(k,\sigma) } \propinquity\left(\left(C^\ast\left(Z^d_c,\theta\right),\Lip_{l,c,\theta}\right),\left(C^\ast\left(\Z^d_k,\sigma\right),\Lip_{l,k,\sigma}\right)\right) = 0\text{,}
\end{equation*}
if $k \in \N_\ast^d$, by replacing $\omega_{N,M}$ with the projection on $\operatorname{span}\{ e_n : n \in I_k\}$. In fact, the proof is somewhat simpler, as in fact it takes place on the finite dimensional $\operatorname{span}\{ e_n : n \in I_k\}$ and does not require Theorem (\ref{qt-commutator-thm}).

We observe that the use of a bridge of the form $(\B(\ell^2(\Z^d)),\omega,\pi,\rho)$ with $\omega$ a compact operator was essential to take advantage of the strong-operator-topology continuity given by the field of quantum and fuzzy tori. This illustrates the potential to apply the quantum propinquity to a large class of examples. The study of the topological and metric properties of the quantum Gro\-mov-Haus\-dorff propinquity and of its applications to continuous fields and other examples will be the matter of an upcoming paper.

%%%%%%%%%%%%%%%%%%%%%%%%%%%%%%%%%%%%%%%%%%%%%%%%%%%%%%%%%%%%%%%%%%%%%%%%%%%%%%%%%%%%%%%%%%%%%%%%%%%%%%%%%

%\backmatter

\bibliographystyle{amsplain}
\bibliography{../thesis}

\vfill
\end{document}